\documentclass[pdflatex,sn-mathphys-num]{sn-jnl}

\usepackage{graphicx}%
\usepackage{multirow}%
\usepackage{amsmath,amssymb,amsfonts}%
\usepackage{amsthm}%
\usepackage{mathrsfs}%
\usepackage[title]{appendix}%
\usepackage{xcolor}%
\usepackage{textcomp}%
\usepackage{manyfoot}%
\usepackage{booktabs}%
\usepackage{algorithm}%
\usepackage{algorithmicx}%
\usepackage{algpseudocode}%
\usepackage{listings}%

\theoremstyle{thmstyleone}%
\newtheorem{theorem}{Theorem}
\newtheorem{proposition}[theorem]{Proposition}%

\theoremstyle{thmstyletwo}%
\newtheorem{remark}{Remark}%

\newtheorem{corollary}[theorem]{Corollary}

\theoremstyle{thmstylethree}%

\begin{document}

\title[Article Title]{Hodge Spectral Surrogates for Topology-Constrained Optimization}


\author*[1]{\fnm{Satoshi} \sur{Kanno}}\email{satoshi.kanno06@g.softbank.co.jp}

\author[1]{\fnm{Yoshi-aki} \sur{Shimada}}\email{yoshiaki.shimada01@g.softbank.co.jp}

\affil*[1]{
\orgdiv{Quantum Information Technology Department, Quantum Technology Division, Product Research and Development Division},
\orgname{SoftBank Corp.},
\orgaddress{
\street{1-7-1 Kaigan},
\city{Minato-ku},
\postcode{105-7529},
\state{Tokyo},
\country{Japan}
}
}


\abstract{
Topology-constrained optimization arises in applications such as medical image segmentation, geometric design, and network generation, where global structural correctness matters. Yet Betti numbers and persistent homology are discrete and combinatorial, making them difficult to control directly by gradient-based optimization. We propose a differentiable framework that replaces discrete topological counting with a soft spectral relaxation of the Hodge Laplacian. Candidate simplices are embedded in a fixed ambient complex, while simplex inclusion is represented by continuous activations, producing a smooth spectral path between different hard complexes. The soft states are used only as optimization surrogates; in the hard limit, the ordinary Hodge Laplacian and its homological zero modes are recovered. We establish operator-norm and spectral perturbation bounds connecting hard Betti numbers to soft near-zero modes, and derive differentiable objectives using heat, resolvent, and polynomial low-pass filters. Controlled experiments on Vietoris--Rips point clouds show more spatially distributed gradients, reduced scale-normalized derivative variation near persistence-pairing changes, and geometry-aware update directions in the tested settings. Experiments on graph clique complexes further show controllable shifts in sampled hard Betti regimes and compatibility with standard graph objectives. The framework provides a complementary Hodge-spectral approach to persistent-homology-based topology optimization.
}




\maketitle

\section{Introduction}

Data such as images, point clouds, and graphs often contain global topological structures, including connected components, loops, and cavities. Such structures can be quantified by Betti numbers and persistent homology and have been widely used as fundamental descriptors in topological data analysis (TDA)\cite{edelsbrunner2002topological,zomorodian2004computing,cohen2005stability,ghrist2008barcodes,edelsbrunner2008persistent,carlsson2009topology,niyogi2008finding,de2007coverage,otter2017roadmap,obayashi2022persistent,nakamura2015persistent,hiraoka2016hierarchical,adams2017persistence,hofer2017deep,chazal2014stochastic,kalivsnik2019tropical}. In recent years, however, increasing attention has been paid not only to analyzing topology but also to topology-constrained optimization, in which generated or estimated objects are optimized so as to possess desired topological properties. In image segmentation, for example, high pixel-wise accuracy does not necessarily guarantee topological correctness: structures that should be connected may become disconnected, while spurious connected components or holes may appear. Such issues are particularly important in medical imaging applications involving blood vessels, the heart, the placenta, and other anatomical structures, and methods that incorporate known topological priors into loss functions have already been investigated \cite{clough2020topological,hu2019topology,shit2021cldice,byrne2022persistent,mosinska2018beyond,stucki2023topologically,demir2023topology}. Therefore, stable incorporation of topological constraints into gradient-based optimization is important for a broad range of structural estimation and design problems, including medical image analysis.

Betti numbers themselves, however, are discrete quantities. Even a small change in optimization variables may cause simplices to appear or disappear, thereby changing the underlying homology. Consequently, directly optimizing Betti numbers using standard gradient-based methods is difficult. To address this difficulty, differentiable topological losses based on the birth--death information of persistent homology have been extensively studied \cite{poulenard2018topological,chen2019topological,hu2019topology,nigmetov2024topological,leygonie2022framework,carriere2021optimizing,clough2020topological,gabrielsson2020topology,moor2020topological,vandaele2021topologically,hofer2020graph,solomon2021fast}. These methods are powerful and have been successfully used to incorporate topological priors, particularly in image and medical-image segmentation. Nevertheless, gradients of persistence-based losses are propagated through the critical simplices that determine births and deaths and may therefore become localized around a relatively small number of simplices\cite{nigmetov2024topological,leygonie2022framework}. Moreover, changes in persistence pairings may induce corresponding changes in the optimization signal. In addition, although persistence diagrams provide a compact representation of the birth--death information of homological features, they do not explicitly retain the surrounding geometric or spectral structure.

In this work, we approach the problem from a different spectral viewpoint. According to Hodge theory, the $q$-th Hodge Laplacian $L_q(K)$ of a simplicial complex $K$ satisfies
\[
\dim \ker L_q(K) = \beta_q(K).
\]
Thus, Betti numbers can be represented spectrally as the number of zero modes of the Hodge Laplacian. This relation naturally suggests that, instead of discretely counting zero eigenvalues, one may smoothly emphasize zero and near-zero spectral modes and use them as continuous optimization signals for homological structure.

Hodge Laplacians and spectral filtering based on them are not new. Weighted Hodge Laplacians\cite{horak2013spectra,schaub2020random,baccini2022weighted}, higher-order spectral analysis\cite{krishnagopal2021spectral,schaub2020random,baccini2022weighted}, and simplicial signal processing\cite{barbarossa2020topological,schaub2021signal,yang2022simplicial,isufi2022convolutional} have already developed methods for analyzing and filtering simplex-weighted structures and Hodge spectra. Persistent Laplacians likewise provide a spectral framework for describing persistent topological information associated with a given filtration \cite{memoli2022persistent,wei2025persistent,wang2020persistent}. However, many of these approaches primarily focus on analyzing or processing spectral quantities and signals over a fixed simplicial structure.

In particular, conventional weighted Hodge formulations allow simplex weights to vary continuously over a fixed simplicial complex. However, the ability to vary an operator continuously is not equivalent to the ability to vary topology continuously. As long as the underlying simplicial support is fixed and all weights remain positive, continuously varying the weights does not change the homology of the underlying chain complex \cite{guglielmi2023quantifying,ren2021weighted}. Thus, conventional weighted Hodge frameworks naturally describe
\[
\text{fixed topology} + \text{continuous weights},
\]
but they do not directly provide a continuous representation for moving between hard complexes with different topologies using gradient-based optimization.

Our approach is therefore to avoid rebuilding the simplicial complex discretely during optimization. Instead, we fix in advance an ambient complex containing all simplices that may appear during optimization, and softly relax the presence or absence of simplices within this common ambient space. Rather than immediately deciding whether each simplex is present or absent by a binary value of zero or one, we represent its activation by a continuous weight. In this way, different hard complexes can be embedded into the same ambient space and interpolated through a continuous spectral path of the form
\[
\text{hard complex A}
\;\longrightarrow\;
\text{soft intermediate states}
\;\longrightarrow\;
\text{hard complex B}.
\]

An important point is that the soft intermediate states are not interpreted as defining a new exact topology. In the soft regime, we use the construction as a Hodge-spectral relaxation and employ its low-frequency spectral structure as an optimization signal. On the other hand, the formulation is constructed so that, in the hard limit where simplex activations approach zero or one, the ordinary simplicial complex and its Hodge Laplacian are recovered. Thus, our framework does not attempt to convert topology itself into a continuous quantity. Rather, it embeds distinct discrete topologies into a common ambient space and continuously interpolates between them through a soft spectral representation obtained by relaxing simplex inclusion. This makes it possible to treat inherently discrete topological changes within a continuous optimization framework.

Applying low-pass filters to this spectral relaxation yields differentiable topological objectives based on zero and near-zero modes. In this work, we consider a spectral-filter matching loss that compares target low-frequency spectral structure and a trace-type loss that controls low-frequency spectral mass. We apply these objectives to point-cloud Vietoris--Rips complexes and graph clique complexes and investigate their optimization behavior. In comparisons with persistent-homology-based losses, we examine not only the relation to hard topology but also the spatial distribution of gradients, behavior near persistence-pairing transitions, and update directions associated with geometric deformations. The aim of this work is not to replace persistent homology, but rather to provide a complementary Hodge-spectral optimization signal for topology-constrained optimization.

The spectral viewpoint adopted here also has a conceptual connection to quantum topological data analysis \cite{lloyd2016quantum,gyurik2022towards,ameneyro2024quantum,hayakawa2022quantum,yamauchi2025quantum}. Quantum TDA approaches encode simplicial complexes in larger Hilbert spaces and estimate Betti-type quantities through the kernel or low-frequency spectral subspace of combinatorial Laplacians \cite{ubaru2021quantumtopologicaldataanalysis,akhalwaya2024comparing}. In this sense, the basic viewpoint of representing a complex in a fixed ambient space and extracting topological information from the Laplacian spectrum is shared with the present framework. However, this work does not propose a quantum algorithm or claim a quantum advantage. The possible use of quantum spectral estimation for polynomial filtering, trace estimation, or gradient evaluation is regarded only as a potential direction for future computation.

The main contributions of this work are summarized as follows.

\begin{enumerate}
    \item \textbf{We construct a Hodge-spectral relaxation that allows changing simplicial complexes to be treated continuously within a fixed ambient space.}
    Simplex inclusion itself is represented softly, while the ordinary Hodge Laplacian and homology are recovered in the hard limit.

    \item \textbf{We construct differentiable topology-control objectives based on the low-frequency Hodge spectrum.}
    Discrete zero-mode counting is replaced by smooth spectral surrogates suitable for gradient-based optimization.

    \item \textbf{We investigate the properties of the proposed optimization signals using point clouds and graph clique complexes.}
    Through controlled numerical experiments, we evaluate their relation to hard topology, their differences from persistent-homology-based losses, and their compatibility with conventional graph objectives.
\end{enumerate}

In summary, this work does not propose the Hodge Laplacian or spectral filtering itself as a new technique. Rather, it provides a framework for incorporating homological constraints into continuous optimization by embedding different hard topologies into a common ambient complex and continuously interpolating between them through a soft spectral representation.

\section{Preliminaries and Problem Setting}

In this section, we summarize the topological quantities considered in this work and the difficulties that arise when using them as loss functions.
We first introduce clique complexes and Betti numbers constructed from graphs, and then describe their relationship with weighted graph Laplacians and classical spectral relaxation.
We subsequently introduce persistent homology for point-cloud data and loss functions based on persistence diagrams, and finally summarize the main issues that arise when such losses are used for optimization.

\subsection{Clique Complexes and Betti Numbers}

We first introduce the Betti numbers of clique complexes as topological quantities for graph data.
Let \(G=(V,E)\) be a graph, where \(V\) is the set of vertices and \(E\) is the set of edges.
The clique complex \(X(G)\) of \(G\) is a simplicial complex obtained by regarding each clique in the graph as a simplex.
That is, a subset of vertices
\[
\sigma=\{v_0,\ldots,v_k\}\subset V
\]
is a \(k\)-simplex if, for any distinct \(i,j\),
\[
(v_i,v_j)\in E
\]
holds, or equivalently, if \(\sigma\) forms a \((k+1)\)-clique.
Thus,
\[
X(G)=\{\sigma\subset V:\sigma \text{ is a clique in }G\}.
\]

Let \(C_k(X(G))\) be the \(k\)-th chain group of \(X(G)\), and let
\[
\partial_k:C_k(X(G))\to C_{k-1}(X(G))
\]
be the boundary operator.
The \(k\)-th homology group is defined by
\[
H_k(X(G))
=
\ker\partial_k/\operatorname{im}\partial_{k+1},
\]
and its dimension
\[
\beta_k(X(G))=\dim H_k(X(G))
\]
is called the \(k\)-th Betti number.
In particular, \(\beta_0\) represents the number of connected components, while \(\beta_1\) represents the number of independent one-dimensional cycles that are not filled by triangles or higher-dimensional simplices.
The Betti numbers of clique complexes describe connectedness, cyclic structure, and higher-order relations in graphs.
Therefore, controlling these quantities in graph generation or network design provides one way to control redundancy and cyclic structure in networks.

\subsection{Weighted Graph Laplacians and Spectral Relaxation}

The idea of describing discrete graph structure through continuous spectral quantities has been widely used in spectral graph theory\cite{chung1997spectral,vonluxburg2007tutorialspectralclustering,lee2014multiwayspectralpartitioninghigherorder}.
Let \(B_1\) be the incidence matrix of a weighted graph, and let
\[
W_1=\operatorname{diag}(w_e)
\]
be the diagonal matrix of edge weights.
The weighted combinatorial graph Laplacian can then be written as
\[
L_0=B_1W_1B_1^{\top}.
\]

The multiplicity of the zero eigenvalue of \(L_0\) equals the number of connected components of the graph, namely \(\beta_0\).
Moreover, the smallest positive eigenvalues are related to graph connectivity and decomposability.
Through algebraic connectivity and Cheeger-type inequalities, they can be related to sparse cuts and weakly coupled structures
(see, e.g., \cite{chung1997spectral,vonluxburg2007tutorialspectralclustering,lee2014multiwayspectralpartitioninghigherorder}).

Thus, in spectral graph theory, discrete connectivity and graph partitioning problems can be continuously relaxed using the eigenvalues and eigenvectors of graph Laplacians.
In this work, we extend this viewpoint to higher-order homological structure and use the low-frequency spectra of Hodge-type operators as differentiable surrogate quantities.

\subsection{Difficulties in Using Betti Numbers as Loss Functions}

Suppose that we want to optimize a graph so that the Betti number of its clique complex approaches a prescribed target value.
For a target Betti number \(\beta_k^{\mathrm{tar}}\), one may consider, for example,
\[
L_{\beta}(G)
=
\left(
\beta_k(X(G))-\beta_k^{\mathrm{tar}}
\right)^2.
\]

However, directly optimizing this loss using standard gradient-based methods is difficult.
The reason is that the map
\[
G
\longmapsto
X(G)
\longmapsto
\beta_k(X(G))
\]
is combinatorial and discrete.
Even a small change in the presence or absence of edges can discontinuously change the set of simplices contained in the clique complex, which in turn can change the ranks of the boundary operators and the dimensions of the homology groups.

Therefore, rather than optimizing Betti numbers directly, it is necessary to construct continuous and differentiable surrogate quantities that retain a clear relationship with homology.

\subsection{Persistent Homology}

Next, we introduce persistent homology
\cite{edelsbrunner2002topological,zomorodian2004computing,cohen2005stability,edelsbrunner2008persistent,otter2017roadmap}
as a topological quantity for point-cloud data and related data types.

Suppose that, from data \(X\), we construct a sequence of simplicial complexes indexed by scale parameters
\[
r_0<r_1<\cdots<r_m
\]
such that
\[
K_{r_0}
\subset
K_{r_1}
\subset
\cdots
\subset
K_{r_m}.
\]
Such a nested sequence of simplicial complexes is called a filtration.
For example, for point-cloud data, one can use Vietoris--Rips complexes indexed by a distance scale \(r\).

At each scale \(r\), we consider
\[
H_k(K_r)
=
\ker\partial_k/\operatorname{im}\partial_{k+1}.
\]
Persistent homology tracks when each homology class is born and when it dies along the filtration.
For a homology class \(\gamma\), let \(b_\gamma\) denote its birth time and \(d_\gamma\) its finite death time.
The \(k\)-th persistence diagram is represented as
\[
\operatorname{Dgm}_k(X)
=
\{(b_\gamma,d_\gamma)\}_{\gamma}
\subset\mathbb{R}^2.
\]
The persistence of \(\gamma\) is defined by
\[
\operatorname{pers}(\gamma)
=
d_\gamma-b_\gamma.
\]
Persistence diagrams and barcode diagrams provide compact summaries of the multiscale topological structure of data and are widely used for shape analysis of point clouds, images, time-series data, and related objects.

\subsection{Loss Functions Based on Persistence Diagrams}

One way to incorporate persistent homology into optimization is to construct loss functions from persistence diagrams or barcode diagrams.
Given a target persistence diagram \(D^{\mathrm{tar}}\), one may define
\[
L_{\mathrm{PD}}(X)
=
\ell
\left(
\operatorname{Dgm}_k(X),
D^{\mathrm{tar}}
\right).
\]
For example, if \(U\) is the set of undesired homology classes, one may use
\[
L_{\mathrm{PD}}(X)
=
\sum_{\gamma\in U}
(d_\gamma-b_\gamma)^2,
\]
which acts to shorten the lifetimes of those classes.
On the other hand, if \(T\) is the set of homology classes to be preserved and \(\tau_\gamma\) is the target persistence of class \(\gamma\), one may consider
\[
L_{\mathrm{PD}}(X)
=
\sum_{\gamma\in T}
\left(
(d_\gamma-b_\gamma)-\tau_\gamma
\right)^2.
\]

In a local region where the relevant persistence pairing remains fixed and the filtration values are nondegenerate, such losses can be differentiated through the birth and death values.
Let \(\sigma_b^\gamma\) and \(\sigma_d^\gamma\) denote the critical simplices that determine the birth and death of a homology class \(\gamma\), respectively.
Then
\[
b_\gamma=f(\sigma_b^\gamma),
\qquad
d_\gamma=f(\sigma_d^\gamma)
\]
can be written in terms of the filtration function \(f\).
Accordingly, the gradient can be expressed schematically as
\[
\nabla_X L_{\mathrm{PD}}
=
\sum_\gamma
\left[
\frac{\partial L_{\mathrm{PD}}}{\partial b_\gamma}
\nabla_Xf(\sigma_b^\gamma)
+
\frac{\partial L_{\mathrm{PD}}}{\partial d_\gamma}
\nabla_Xf(\sigma_d^\gamma)
\right].
\]
Thus, gradients of persistence-based losses propagate through the critical simplices that determine birth and death values.

\subsection{Optimization Issues of Persistence-Based Losses}

Persistence-based losses provide a powerful way to incorporate topology into optimization.
However, when they are used as optimization signals, the following properties should be taken into account.

\subsubsection{Gradient Localization}

The first issue is that gradients tend to be localized on a small number of critical simplices
(see \cite{nigmetov2024topological,nigmetov2024topological_2}).

As described above, the birth and death of each homology class are determined by critical simplices
\(\sigma_b^\gamma\) and \(\sigma_d^\gamma\).
Therefore, updates mainly propagate through
\[
\nabla_Xf(\sigma_b^\gamma),
\qquad
\nabla_Xf(\sigma_d^\gamma),
\]
and hence tend to be concentrated on the vertices contained in these simplices and their neighborhoods.
On the other hand, topological structures are determined by the global configuration of the data.
For example, when a point cloud forms a large loop, the birth or death of the loop may be represented by only a small number of critical simplices, even though the geometry of the entire loop is important.
Thus, when update directions reflecting a broader geometric structure are desirable, gradient localization may become a limitation.

\subsubsection{Nonsmoothness under Pairing Changes}

The second issue is that the gradient direction can change nonsmoothly or abruptly when persistence pairings switch due to small changes in filtration values.

Let
\[
(\sigma_b,\sigma_d)
\]
be a birth--death pair associated with a homology class.
A small perturbation of the input \(X\) can change the ordering of simplex filtration values, resulting in a change of pairing of the form
\[
(\sigma_b,\sigma_d)
\longrightarrow
(\sigma'_b,\sigma'_d).
\]
In this case, the simplices through which gradients propagate also switch as
\[
\nabla_Xf(\sigma_b),
\quad
\nabla_Xf(\sigma_d)
\longrightarrow
\nabla_Xf(\sigma'_b),
\quad
\nabla_Xf(\sigma'_d).
\]
Therefore, even when the persistence diagram or the loss value changes only slightly, the update direction may change abruptly near a pairing transition.
This property is important when gradient-based optimization is considered.

\subsubsection{Limited Encoding of Geometric and Spectral Information}

The third issue is that persistence diagrams are representations specialized for summarizing the birth--death information of homology classes and do not explicitly encode surrounding geometric or spectral information.

A persistence diagram is obtained from the input data \(X\) through the map
\[
X
\longmapsto
\operatorname{Dgm}_k(X)
=
\{(b_\gamma,d_\gamma)\}_{\gamma}.
\]
This representation compactly describes the scales at which each homology class is born and dies.
However, it does not directly contain the spatial locations of the points or simplices supporting the homology classes, the surrounding neighborhood structure, or spectral information such as Laplacian eigenvectors and low-frequency modes.
For example, two homology classes may have the same persistence while differing in their arrangement in the data space, the geometric shape of the corresponding cycles, or the properties of near-closed structures.
Thus, while persistence-based losses are natural for controlling topological birth--death information, there remains room to exploit additional information when geometrically distributed and smooth update directions are desired.

\subsection[Problem Addressed in This Work]{Problem Addressed in This Work}

The preceding discussion identifies two different uses of topology.
The first is descriptive: Betti numbers or persistent homology are computed for fixed data in order to characterize global structure.
The second is prescriptive: a point cloud, graph, or model output is optimized so that it possesses a desired topology.
This work focuses on the latter setting.

For clique complexes of graphs, the map
\[
G
\longmapsto
X(G)
\longmapsto
\beta_k(X(G))
\]
is combinatorial, because the set of simplices and the ranks of the boundary operators can change discontinuously when edges are added or removed.
On the other hand, persistence-based losses for Vietoris--Rips filtrations allow differentiation through birth and death values, but exhibit gradient localization, nonsmoothness near pairing transitions, and the property that geometric and spectral information is not explicitly used.

\section{Construction of Hodge-Laplacian-type Spectral Relaxation}

In this section, we introduce a Hodge-Laplacian-type spectral relaxation for constructing differentiable topological objectives for graphs and point clouds.

The basic idea of this work is to embed simplicial complexes that vary combinatorially during optimization into a fixed ambient chain space and to replace hard simplex inclusion with continuous simplex weights. The entire construction can be represented by the sequence
\[
\theta
\longmapsto
a_e(\theta)
\longmapsto
p_e^{(\varepsilon)}(\theta)
\longmapsto
w_\sigma^{(\varepsilon)}(\theta)
\longmapsto
\widehat L_q^{(\varepsilon)}(\theta).
\]
Here, \(a_e(\theta)\) is a continuous parameter associated with an edge, \(p_e^{(\varepsilon)}\) is a soft edge activation, \(w_\sigma^{(\varepsilon)}\) is a simplex weight, and \(\widehat L_q^{(\varepsilon)}\) is a Hodge-Laplacian-type spectral relaxation. A common parameter \(\varepsilon>0\) controls the width of the softened hard threshold.

An important point is that, in the soft regime, \(\widehat L_q^{(\varepsilon)}\) is not regarded as the Hodge Laplacian of a new exact chain complex. Instead, we interpret it as a spectral operator obtained by smoothly perturbing the Hodge Laplacian of the corresponding hard complex.
Below, we quantify this soft-to-hard relationship in operator norm and show that the zero modes corresponding to the homology of the hard complex appear as near-zero low-eigenvalue modes of the soft operator.

\subsection{Fixed Ambient Complex and Hard Hodge Operator}

Let \(K_{\max}\) be a finite simplicial complex containing all candidate simplices that may appear during optimization. For example, given a candidate edge set \(E_{\max}\), one may use the clique complex
\[
K_{\max}=X(G_{\max})
\]
of the graph $G_{\max}=(V,E_{\max})$ as the ambient complex.
For each degree \(q\), define the fixed ambient chain space by
\[
C_q^{\mathrm{amb}}
=
C_q(K_{\max}).
\]
After fixing orientations, this is a finite-dimensional real vector space whose basis consists of all \(q\)-simplices contained in \(K_{\max}\).
Let \(B_q\) denote the matrix representation of the boundary operator of \(K_{\max}\). Then $B_qB_{q+1}=0$.

For a parameter \(\theta\), an edge \(e\) is defined to be active when $a_e(\theta)>0$.
Let $K(\theta)\subset K_{\max}$
be the clique complex generated by the active edge set.
For each \(q\)-simplex \(\sigma\), define its hard activation indicator by
\[
\chi_\sigma(\theta)
=
\mathbf 1_{\{\sigma\in K(\theta)\}},
\]
and define
\[
\Pi_q(\theta)
=
\operatorname{diag}
\left(
\chi_\sigma(\theta)
\right).
\]
Then
\[
C_q(K(\theta))
=
\operatorname{Im}\Pi_q(\theta)
\subset
C_q^{\mathrm{amb}}.
\]
Since \(K(\theta)\) is a simplicial complex, every face of an active simplex is also active.
We define the hard restricted boundary operator on the ambient space by
\[
B_q^{\mathrm{hard}}
=
\Pi_{q-1}B_q\Pi_q.
\]
On the active subspace, this agrees with the ordinary boundary operator of \(K(\theta)\).
On the other hand, if one constructs a Hodge operator using only the boundary terms, inactive simplex directions remain as artificial zero modes. Therefore, for \(\mu>0\), we define
\[
L_q^{\mathrm{hard}}
=
(B_q^{\mathrm{hard}})^\top B_q^{\mathrm{hard}}
+
B_{q+1}^{\mathrm{hard}}
(B_{q+1}^{\mathrm{hard}})^\top
+
\mu(I-\Pi_q).
\]

The last term assigns eigenvalue \(\mu\) to inactive \(q\)-simplex directions and removes them from the zero eigenspace.

\begin{proposition}[Consistency of the hard ambient operator]
With respect to the orthogonal decomposition
\[
C_q^{\mathrm{amb}}
=
C_q(K(\theta))
\oplus
C_q(K(\theta))^\perp,
\]
one has
\[
L_q^{\mathrm{hard}}
=
L_q(K(\theta))
\oplus
\mu I_{C_q(K(\theta))^\perp}.
\]
Consequently,
\[
\ker L_q^{\mathrm{hard}}
\simeq
H_q(K(\theta)),
\]
and
\[
\dim\ker L_q^{\mathrm{hard}}
=
\beta_q(K(\theta)).
\]
\end{proposition}

\begin{proof}
On the active subspace, the projected boundary operators coincide with the ordinary boundary operators of \(K(\theta)\). On the inactive orthogonal complement, the boundary terms vanish and only the penalty term acts as \(\mu I\). Hence the stated direct-sum decomposition follows.
Furthermore, by the finite-dimensional Hodge decomposition,
\[
\ker L_q(K(\theta))
\simeq
H_q(K(\theta)).
\]
Therefore, the multiplicity of the zero eigenvalue is equal to \(\beta_q(K(\theta))\).
\end{proof}

\subsection{Differentiable Soft Relaxation}

We next replace the hard edge indicators by continuous activation weights.
For each candidate edge \(e\in E_{\max}\), define
\[
p_e^{(\varepsilon)}(\theta)
=
\sigma\left(
\frac{a_e(\theta)}{\varepsilon}
\right),
\qquad
\sigma(z)=\frac{1}{1+e^{-z}},
\]
where the common parameter \(\varepsilon>0\) controls the width of the soft threshold.
The weight of a simplex \(\sigma\) is defined by
$w_\sigma^{(\varepsilon)}
=
\prod_{e\subset\sigma}
p_e^{(\varepsilon)}$.
Vertices are assumed to always be present, so that $w_v^{(\varepsilon)}=1$.
This product construction is a continuous relaxation of the logical conjunction appearing in the clique condition.
For each degree \(q\), define
\[
W_q^{(\varepsilon)}
=
\operatorname{diag}
\left(
w_\sigma^{(\varepsilon)}
\right),
\qquad
R_q^{(\varepsilon)}
=
\left(W_q^{(\varepsilon)}\right)^{1/2}.
\]
We define the weighted boundary-type operator by
\[
\widetilde B_q^{(\varepsilon)}
=
R_{q-1}^{(\varepsilon)}
B_q
R_q^{(\varepsilon)},
\]
and the Hodge-Laplacian-type spectral relaxation by
\[
\widehat L_q^{(\varepsilon)}
=
(\widetilde B_q^{(\varepsilon)})^\top
\widetilde B_q^{(\varepsilon)}
+
\widetilde B_{q+1}^{(\varepsilon)}
(\widetilde B_{q+1}^{(\varepsilon)})^\top
+
\mu
\left(
I-W_q^{(\varepsilon)}
\right).
\]

In the general soft regime, $\widetilde B_{q-1}^{(\varepsilon)}
\widetilde B_q^{(\varepsilon)}
\neq0$,
and hence \(\widetilde B_q^{(\varepsilon)}\) is not the boundary operator of an exact chain complex.
Accordingly, \(\widehat L_q^{(\varepsilon)}\) is not regarded as the Hodge Laplacian of a new simplicial complex, but rather as a differentiable spectral relaxation of the hard operator \(L_q^{\mathrm{hard}}\).

Since vertices are always active,
$W_0^{(\varepsilon)}
=
R_0^{(\varepsilon)}
=
I$.
Therefore,
\[
\widehat L_0^{(\varepsilon)}
=
B_1W_1^{(\varepsilon)}B_1^\top.
\]
This is exactly the ordinary weighted combinatorial graph Laplacian with edge weights \(p_e^{(\varepsilon)}\).
Thus, at degree zero, the proposed construction coincides with classical spectral graph theory, while the higher-order operators \(\widehat L_q^{(\varepsilon)}\) can be regarded as its natural extension to higher-order structures.

\subsection{Quantitative Relation Between the Soft and Hard Operators}

We now quantify the extent to which \(\widehat L_q^{(\varepsilon)}\) can be regarded as a blurred version of the corresponding hard operator \(L_q^{\mathrm{hard}}\).
In what follows, we consider a regular parameter value \(\theta\) at which the hard combinatorial structure is fixed. That is, we assume that $a_e(\theta)\neq0$
for every candidate edge.
Define the minimum margin from the hard threshold by
\[
m(\theta)
=
\min_{e\in E_{\max}}
|a_e(\theta)|
>0.
\]
By the properties of the sigmoid function, for every edge,
\[
\left|
p_e^{(\varepsilon)}
-
\mathbf 1_{\{a_e>0\}}
\right|
\le
e^{-m(\theta)/\varepsilon}.
\]
Thus, as \(\varepsilon\) decreases, the soft edge activation approaches the corresponding hard indicator exponentially fast.
Let $c_j=\binom{j+1}{2}$
denote the number of edges contained in a \(j\)-simplex, with \(c_0=0\).
For every \(j\)-simplex,
\[
\left|
w_\sigma^{(\varepsilon)}
-
\chi_\sigma
\right|
\le
c_j e^{-m(\theta)/\varepsilon},\quad\quad
\left\|
W_j^{(\varepsilon)}-\Pi_j
\right\|_2
\le
c_j e^{-m(\theta)/\varepsilon}.
\]
Furthermore, since \(0\le W_j^{(\varepsilon)}\le I\),
\[
\left\|
R_j^{(\varepsilon)}-\Pi_j
\right\|_2
\le
\sqrt{c_j}\,
e^{-m(\theta)/(2\varepsilon)}.
\]
These estimates allow us to directly control the difference between the soft and hard Hodge operators.

\begin{proposition}[Quantitative soft-to-hard operator approximation]
There exists a constant \(C_q>0\) such that
\[
\left\|
\widehat L_q^{(\varepsilon)}
-
L_q^{\mathrm{hard}}
\right\|_2
\le
C_q
e^{-m(\theta)/(2\varepsilon)}
+
\mu c_q
e^{-m(\theta)/\varepsilon}
\]
holds.
Here, \(C_q\) is independent of \(\varepsilon\) and \(\theta\), and depends only on the boundary matrices of the ambient complex and on the degree \(q\).
For example, one may take
\[
C_q
=
2\|B_q\|_2^2
\left(
\sqrt{c_{q-1}}+\sqrt{c_q}
\right)
+
2\|B_{q+1}\|_2^2
\left(
\sqrt{c_q}+\sqrt{c_{q+1}}
\right).
\]
\end{proposition}

\begin{proof}
First, decompose the difference between the soft and hard boundary operators as
\[
\begin{aligned}
\widetilde B_q^{(\varepsilon)}
-
B_q^{\mathrm{hard}}
={}&
\left(
R_{q-1}^{(\varepsilon)}-\Pi_{q-1}
\right)
B_qR_q^{(\varepsilon)}
\\
&+
\Pi_{q-1}B_q
\left(
R_q^{(\varepsilon)}-\Pi_q
\right).
\end{aligned}
\]
Since \(0\le R_j^{(\varepsilon)}\le I\) and \(\|\Pi_j\|_2\le1\),
\[
\left\|
\widetilde B_q^{(\varepsilon)}
-
B_q^{\mathrm{hard}}
\right\|_2
\le
\|B_q\|_2
\left(
\|R_{q-1}^{(\varepsilon)}-\Pi_{q-1}\|_2
+
\|R_q^{(\varepsilon)}-\Pi_q\|_2
\right).
\]
Moreover, for arbitrary matrices \(A\) and \(C\),
\[
\|A^\top A-C^\top C\|_2
\le
(\|A\|_2+\|C\|_2)\|A-C\|_2.
\]
Applying this inequality to the lower and upper Hodge terms, together with
\[
\left\|
\mu(I-W_q^{(\varepsilon)})
-
\mu(I-\Pi_q)
\right\|_2
=
\mu
\|W_q^{(\varepsilon)}-\Pi_q\|_2,
\]
gives the stated estimate.
\end{proof}

\begin{corollary}[Hard limit]
The preceding estimate implies
\[
\widehat L_q^{(\varepsilon)}
\longrightarrow
L_q^{\mathrm{hard}}
\qquad
(\varepsilon\to0)
\]
in operator norm.
In particular, in a region with a positive margin \(m(\theta)>0\) from the hard threshold, the error vanishes at least at the rate
\[
O\left(
e^{-m(\theta)/(2\varepsilon)}
\right).
\]
\end{corollary}

Thus, the soft operator is not merely formally related to the hard operator in the limit. For finite \(\varepsilon\), it can be interpreted as a spectral perturbation of the hard operator whose width is explicitly controlled.
For convenience, define this perturbation width by
\[
\rho_q(\varepsilon,\theta)
=
C_q
e^{-m(\theta)/(2\varepsilon)}
+
\mu c_q
e^{-m(\theta)/\varepsilon}.
\]
Then
\[
\left\|
\widehat L_q^{(\varepsilon)}
-
L_q^{\mathrm{hard}}
\right\|_2
\le
\rho_q(\varepsilon,\theta),
\]
and
$\rho_q(\varepsilon,\theta)
\longrightarrow0
~
(\varepsilon\to0)$.
If, for some edge,
$a_e(\theta)=0$,
then \(m(\theta)=0\), and this local estimate degenerates. This is precisely the point at which the combinatorial structure of the corresponding hard complex changes.
Thus, the preceding result gives a quantitative soft-to-hard stability estimate within each region of parameter space in which the hard combinatorial structure remains fixed.

\begin{remark}[Topology-changing boundaries]
The condition $m(\theta)>0$ is required for the quantitative comparison
with a fixed hard complex. If $a_e(\theta)=0$ for some candidate edge
$e$, the associated hard complex changes combinatorially, and the
perturbation estimate in Proposition~2 no longer yields a small
soft-to-hard error bound with respect to either of the adjacent hard
complexes. Consequently, the spectral localization and Betti-number
approximation results derived from this perturbation bound should not be
interpreted as providing a hard-topology certificate exactly at such a
transition.

On the other hand, for every fixed $\varepsilon>0$, the soft activation
$p_e^{(\varepsilon)}(\theta)$ is smooth in $a_e(\theta)$, and therefore
the soft operator $\widehat{L}_q^{(\varepsilon)}(\theta)$ and the
spectral objectives constructed from it remain continuous across a
topology-changing boundary and are differentiable whenever the
underlying parameterization $a_e(\theta)$ is differentiable.
Thus, the finite-soft formulation can provide an optimization path
across a boundary at which the corresponding hard topology changes,
even though the local perturbation-based hard-topology certificate
temporarily becomes noninformative. Once the iterate enters a region
with $m(\theta)>0$, the preceding soft-to-hard perturbation bounds again
quantify its relation to the corresponding hard complex.
\end{remark}

\subsection{Localization of Hard Zero Modes in the Soft Low-Frequency Spectrum}

The operator-norm estimate above allows us to directly quantify how softening blurs the Hodge spectrum.
Let \(K\) denote the corresponding hard complex and set $b=\beta_q(K)$.
Order the eigenvalues of the hard ambient operator increasingly as
\[
0
=
\lambda_1^{\mathrm{hard}}
=
\cdots
=
\lambda_b^{\mathrm{hard}}
<
\lambda_{b+1}^{\mathrm{hard}}
\le
\cdots.
\]
Let
$\gamma_q
=
\lambda_{b+1}^{\mathrm{hard}}
>0$
denote its smallest positive eigenvalue.
The quantity \(\gamma_q\) is the spectral gap separating the homological zero modes from the remaining spectrum of the hard operator.

\begin{corollary}[Localization of homological zero modes]
Suppose that
\[
\left\|
\widehat L_q^{(\varepsilon)}
-
L_q^{\mathrm{hard}}
\right\|_2
\le
\rho_q.
\]
By Weyl's eigenvalue perturbation inequality,
\[
\left|
\lambda_i(\widehat L_q^{(\varepsilon)})
-
\lambda_i(L_q^{\mathrm{hard}})
\right|
\le
\rho_q
\]
for every \(i\).
Consequently, the soft eigenvalues corresponding to the \(b\) homological zero modes of the hard complex satisfy
\[
0
\le
\lambda_i(\widehat L_q^{(\varepsilon)})
\le
\rho_q,
\qquad
i=1,\ldots,b.
\]
On the other hand, all remaining eigenvalues satisfy
\[
\lambda_i(\widehat L_q^{(\varepsilon)})
\ge
\gamma_q-\rho_q,
\qquad
i>b.
\]
\end{corollary}

In particular, if $\rho_q<\frac{\gamma_q}{2}$,
then the intervals $[0,\rho_q]$
and $[\gamma_q-\rho_q,\infty)$
are disjoint, and \(\widehat L_q^{(\varepsilon)}\) has exactly $\beta_q(K)$
eigenvalues in the interval $[0,\rho_q]$.
Thus, the $\beta_q(K)$ exact zero modes of the hard operator do not disappear under softening. Instead, they spread into
\[
\boxed{
\beta_q(K)\text{ near-zero low-frequency modes}
}
\]
contained in an interval of width \(\rho_q\).
Moreover, a spectral gap of at least $\gamma_q-2\rho_q$
remains between this low-frequency cluster and the rest of the spectrum.
For the common softness parameter \(\varepsilon\),
\[
\rho_q(\varepsilon,\theta)
=
O\left(
e^{-m(\theta)/(2\varepsilon)}
\right),
\]
and therefore the width of the low-frequency cluster shrinks to zero as $\varepsilon\to0$.
Hence, in the operator-norm and eigenvalue-perturbation sense, one obtains the quantitative correspondence
\[
\text{hard zero modes}
\quad\longrightarrow\quad
\text{soft near-zero modes}
\quad\longrightarrow\quad
\text{hard zero modes}.
\]

This result provides the mathematical basis for the low-pass Hodge spectral filters introduced in the next section.
In particular, by choosing a spectral bandwidth that is sufficiently larger than the softening-induced zero-mode spread \(\rho_q\), yet sufficiently smaller than the scale \(\gamma_q-\rho_q\) at which the non-homological modes begin, one can selectively extract the low-frequency modes originating from the homology of the hard complex.

\section{Topological Objectives via Low-Pass Hodge Spectral Filters}

In this section, we construct differentiable topological objectives from the Hodge-Laplacian-type spectral relaxation introduced in the previous section.
We first show that the trace of a low-pass filter provides a quantitative approximation to the Betti number of the corresponding hard complex. We then show that minimizing an appropriate soft spectral loss yields control of the hard Betti number. We also consider objectives that increase or decrease Betti numbers without specifying an exact target. Finally, when a target complex is known, we introduce a matrix-valued spectral-filter matching objective that may retain information beyond the Betti number alone.

All of the objectives introduced below are differentiable through
\(\widehat L_q^{(\varepsilon)}(\theta)\) at regular parameter values. In graph optimization, gradients propagate to edge parameters, whereas in the Vietoris--Rips setting they propagate to point coordinates through pairwise distances. Detailed backward computations are given in the Appendix.

\subsection{Low-Pass Spectral Filters}

Let
$L
=
\widehat L_q^{(\varepsilon)}(\theta)$,
and consider the eigendecomposition
\[
L
=
U\Lambda U^\top,
\qquad
\Lambda
=
\operatorname{diag}
(\lambda_1,\ldots,\lambda_N),
\]
where $N=\dim C_q^{\mathrm{amb}}$.
We call a function $f:[0,\infty)\to[0,1]$
a low-pass spectral filter if it satisfies $f(0)=1$
and is monotonically nonincreasing.
The corresponding matrix function is defined by
\[
f(L)
=
Uf(\Lambda)U^\top.
\]

We mainly use the following two smooth filters.
The heat filter is defined by
\[
f_{\mathrm{heat}}(\lambda)
=
\exp\left(
-\frac{\lambda}{\tau}
\right),
\qquad
\tau>0.
\]
The resolvent filter is defined by
\[
f_{\mathrm{res}}(\lambda)
=
\frac{\alpha}{\lambda+\alpha},
\qquad
\alpha>0.
\]
Both filters take the value one at \(\lambda=0\) and decay as the eigenvalue increases.

The spectral localization result of the previous section suggests that the filter bandwidth should be chosen sufficiently larger than the softening-induced spread \(\rho_q\) of the homological modes, while remaining sufficiently smaller than the scale $\gamma_q-\rho_q$
at which the non-homological modes begin.
In this regime, the filter selectively retains the low-frequency modes originating from the homology of the corresponding hard complex.

\subsection{Low-Frequency Spectral Mass and Hard Betti Numbers}

For a low-pass filter \(f\), define the low-frequency spectral mass by
\[
S_f(\theta)
=
\operatorname{Tr}
f\left(
\widehat L_q^{(\varepsilon)}(\theta)
\right).
\]
The following result relates this soft spectral quantity directly to the Betti number of the corresponding hard complex.

\begin{proposition}[Soft spectral approximation of Betti numbers]
Let \(K\) be the hard complex associated with a soft configuration, and let
$b=\beta_q(K)$.
Suppose that
\[
\left\|
\widehat L_q^{(\varepsilon)}
-
L_q^{\mathrm{hard}}(K)
\right\|_2
\le
\rho_q,
\]
and let $\gamma_q>0$
be the smallest positive eigenvalue of
\(L_q^{\mathrm{hard}}(K)\).
If $0\le\rho_q<\gamma_q$,
then
\[
b f(\rho_q)
\le
S_f
\le
b+(N-b)f(\gamma_q-\rho_q).
\]
Consequently,
\[
\left|
S_f-\beta_q(K)
\right|
\le
E_f,
\]
where
\[
E_f
=
\max
\left\{
b\left[1-f(\rho_q)\right],
\,
(N-b)f(\gamma_q-\rho_q)
\right\}.
\]
\end{proposition}

\begin{proof}
By the spectral perturbation result of the previous section,
\[
0
\le
\lambda_i
\left(
\widehat L_q^{(\varepsilon)}
\right)
\le
\rho_q,
\qquad
i=1,\ldots,b,
\]
whereas
\[
\lambda_i
\left(
\widehat L_q^{(\varepsilon)}
\right)
\ge
\gamma_q-\rho_q,
\qquad
i>b.
\]
Since \(f\) is monotonically nonincreasing and \(f(0)=1\),
\[
f(\rho_q)
\le
f(\lambda_i)
\le
1,
\qquad
i\le b,
\]
and
\[
0
\le
f(\lambda_i)
\le
f(\gamma_q-\rho_q),
\qquad
i>b.
\]
Summing these inequalities over the spectrum gives
\[
b f(\rho_q)
\le
S_f
\le
b+(N-b)f(\gamma_q-\rho_q),
\]
from which the error estimate follows.
\end{proof}

Thus, the low-frequency spectral mass is not merely a heuristic surrogate: its deviation from the Betti number of the corresponding hard complex can be explicitly bounded.

\subsubsection{Heat Filter}

For the heat filter,
\[
f_{\mathrm{heat}}(\lambda)
=
e^{-\lambda/\tau},
\]
the preceding proposition gives
\[
\left|
\operatorname{Tr}
\exp\left(
-\frac{\widehat L_q^{(\varepsilon)}}{\tau}
\right)
-
b
\right|
\le
\max
\left\{
b\left(1-e^{-\rho_q/\tau}\right),
\,
(N-b)e^{-(\gamma_q-\rho_q)/\tau}
\right\}.
\]
The first term measures the attenuation of homological modes caused by their displacement away from zero under softening. The second term measures leakage from non-homological positive-eigenvalue modes through the low-pass filter.
Thus, if the filter scale is chosen so that both error terms are small, heuristically corresponding to the spectrally separated regime
\[
\rho_q
\ll
\tau
\ll
\gamma_q-\rho_q,
\]
then
\[
\operatorname{Tr}
\exp\left(
-\frac{\widehat L_q^{(\varepsilon)}}{\tau}
\right)
\approx
\beta_q(K).
\]

\subsubsection{Resolvent Filter}

For the resolvent filter,
\[
f_{\mathrm{res}}(\lambda)
=
\frac{\alpha}{\lambda+\alpha},
\]
we obtain
\[
\left|
\operatorname{Tr}
\left[
\alpha
\left(
\widehat L_q^{(\varepsilon)}
+
\alpha I
\right)^{-1}
\right]
-
b
\right|
\le
\max
\left\{
b\frac{\rho_q}{\rho_q+\alpha},
\,
(N-b)
\frac{\alpha}
{\gamma_q-\rho_q+\alpha}
\right\}.\]
Hence, when both terms on the right-hand side are small, heuristically corresponding to
\[
\rho_q
\ll
\alpha
\ll
\gamma_q-\rho_q,
\]
the resolvent trace also approximates the hard Betti number.

\subsection{Control of Hard Betti Numbers via Soft Spectral Objectives}

Suppose that a target Betti number $b_{\mathrm{tar}}\in\mathbb N$
is prescribed.

We define the trace-type topological loss by
\[
\mathcal L_{\mathrm{trace}}(\theta)
=
\frac12
\left(
S_f(\theta)-b_{\mathrm{tar}}
\right)^2.
\]
The following result shows that sufficiently accurate minimization of this soft loss implies exact satisfaction of the corresponding hard Betti constraint.

\begin{theorem}[Soft-to-hard consistency of Betti-number control]
Suppose that
\[
\left|
S_f-\beta_q(K)
\right|
\le
E_f.
\]
Then
\[
\left|
\beta_q(K)-b_{\mathrm{tar}}
\right|
\le
E_f
+
\sqrt{
2\mathcal L_{\mathrm{trace}}
}.
\]
In particular, if
\[
E_f
+
\sqrt{
2\mathcal L_{\mathrm{trace}}
}
<
\frac12,
\]
then $\beta_q(K)=b_{\mathrm{tar}}$.
\end{theorem}

\begin{proof}
By the triangle inequality,
\[
\begin{aligned}
\left|
\beta_q(K)-b_{\mathrm{tar}}
\right|
&\le
\left|
\beta_q(K)-S_f
\right|
+
\left|
S_f-b_{\mathrm{tar}}
\right|
\\
&\le
E_f
+
\sqrt{
2\mathcal L_{\mathrm{trace}}
}.
\end{aligned}
\]
Both
\(\beta_q(K)\) and \(b_{\mathrm{tar}}\) are integers. Therefore, if their difference is strictly smaller than \(1/2\), it must be zero.
\end{proof}
This theorem establishes the basic surrogate-consistency principle
\[
\boxed{
\text{soft spectral optimization}
\quad\Longrightarrow\quad
\text{hard Betti-number control}.
}
\]
Thus, provided that the soft-to-hard spectral error is sufficiently small and the spectral mass is optimized sufficiently close to the desired integer value, the associated hard complex is guaranteed to have the prescribed Betti number.

The implication above is a \emph{certification statement for a
configuration whose soft spectrum is sufficiently well separated from
that of the corresponding hard complex}; it is not a statement that the
hard Betti number must vary monotonically, or even be well defined
continuously, along every optimization trajectory. In particular, when
an iterate approaches a topology-changing boundary
$a_e(\theta)=0$, the margin $m(\theta)$ vanishes and the
perturbation-based bound used to control $E_f$ may cease to be
informative. The soft objective itself remains well defined at finite
$\varepsilon$, but hard-topology certification is recovered only after
the iterate moves into a region in which the hard combinatorial
structure is again separated from the threshold.

\subsection{Increasing and Decreasing Betti Numbers}

In some applications, the goal is not to attain a prescribed Betti number but rather to promote or suppress \(q\)-dimensional homological structure.
In this case, an important consequence of the preceding approximation result is that, under a sufficiently small uniform approximation error, the spectral mass preserves the ordering of distinct Betti levels.

Let \(\Omega\) be a parameter region in which
\[
\left|
S_f(\theta)
-
\beta_q(K(\theta))
\right|
\le
\overline E
<
\frac12
\]
holds uniformly.
Consider $\theta_1,\theta_2\in\Omega$.
If
\[
\beta_q(K(\theta_2))
>
\beta_q(K(\theta_1)),
\]
then, since Betti numbers are integers,
\[
\beta_q(K(\theta_2))
\ge
\beta_q(K(\theta_1))+1.
\]
Therefore,
\[
\begin{aligned}
S_f(\theta_2)-S_f(\theta_1)
&\ge
\beta_q(K(\theta_2))
-
\beta_q(K(\theta_1))
-
2\overline E
\\
&\ge
1-2\overline E
>
0.
\end{aligned}
\]
Hence,
\[
\beta_q(K(\theta_2))
>
\beta_q(K(\theta_1))
\quad\Longrightarrow\quad
S_f(\theta_2)
>
S_f(\theta_1).
\]
Thus, within a uniformly spectrally separated region, \(S_f\) preserves the ordering of distinct Betti-number levels.

Assuming that the relevant extrema are attained, it follows that
\[
\operatorname*{arg\,max}_{\theta\in\Omega}
S_f(\theta)
\subseteq
\operatorname*{arg\,max}_{\theta\in\Omega}
\beta_q(K(\theta)),
\]
and similarly,
\[
\operatorname*{arg\,min}_{\theta\in\Omega}
S_f(\theta)
\subseteq
\operatorname*{arg\,min}_{\theta\in\Omega}
\beta_q(K(\theta)).
\]
Consequently, minimizing $-S_f(\theta)$
promotes larger \(q\)-th Betti numbers, whereas minimizing $S_f(\theta)$
suppresses them.

More quantitatively, if
\[
S_f(\theta_2)-S_f(\theta_1)
>
2\overline E,
\]
then
\[
\beta_q(K(\theta_2))
>
\beta_q(K(\theta_1)).
\]
It should be emphasized that this result does not imply that the Betti number increases at every infinitesimal optimization step. Within a region having fixed hard combinatorial structure, \(S_f\) may vary continuously while the Betti number remains constant.
The result instead states that, under a sufficiently accurate and uniform spectral approximation, the soft spectral mass is order-consistent across distinct hard Betti-number levels.

\subsection{Spectral-Filter Matching for a Known Target Complex}

The objectives considered above use only the scalar quantity
\[
S_f
=
\operatorname{Tr}
f(\widehat L_q^{(\varepsilon)}).
\]
Such objectives are appropriate when the primary goal is to control the amount of \(q\)-dimensional homological structure.
When a target graph, point cloud, or simplicial complex is known, however, one can use the entire filtered operator rather than its trace alone.

Let $L_{\mathrm{tar}}$
denote the target operator and define
\[
F(\theta)
=
f
\left(
\widehat L_q^{(\varepsilon)}(\theta)
\right),
\qquad
F_{\mathrm{tar}}
=
f(L_{\mathrm{tar}}).
\]
We define the spectral-filter matching loss by
\[
\mathcal L_{\mathrm{sf}}(\theta)
=
\frac12
\left\|
F(\theta)-F_{\mathrm{tar}}
\right\|_F^2.
\]
For the heat filter,
\[
\mathcal L_{\mathrm{heat}}
=
\frac12
\left\|
\exp
\left(
-\frac{
\widehat L_q^{(\varepsilon)}
}{\tau}
\right)
-
\exp
\left(
-\frac{
L_{\mathrm{tar}}
}{\tau}
\right)
\right\|_F^2.
\]
For the resolvent filter,
\[
\mathcal L_{\mathrm{res}}
=
\frac12
\left\|
\alpha
\left(
\widehat L_q^{(\varepsilon)}
+
\alpha I
\right)^{-1}
-
\alpha
\left(
L_{\mathrm{tar}}
+
\alpha I
\right)^{-1}
\right\|_F^2.
\]
Matrix-valued spectral matching is stronger than trace matching. Indeed,
\[
\left|
\operatorname{Tr}F(\theta)
-
\operatorname{Tr}F_{\mathrm{tar}}
\right|
\le
\sqrt{N}
\left\|
F(\theta)-F_{\mathrm{tar}}
\right\|_F,
\]
and therefore
\[
\left|
\operatorname{Tr}F(\theta)
-
\operatorname{Tr}F_{\mathrm{tar}}
\right|
\le
\sqrt{
2N\mathcal L_{\mathrm{sf}}
}.
\]

Suppose that
\[
\left|
\operatorname{Tr}F(\theta)
-
\beta_q(K(\theta))
\right|
\le
E_{\mathrm{cur}},
\]
and
\[
\left|
\operatorname{Tr}F_{\mathrm{tar}}
-
\beta_q(K_{\mathrm{tar}})
\right|
\le
E_{\mathrm{tar}}.
\]
Then
\[
\left|
\beta_q(K(\theta))
-
\beta_q(K_{\mathrm{tar}})
\right|
\le
E_{\mathrm{cur}}
+
E_{\mathrm{tar}}
+
\sqrt{
2N\mathcal L_{\mathrm{sf}}
}.
\]
Consequently, if
\[
E_{\mathrm{cur}}
+
E_{\mathrm{tar}}
+
\sqrt{
2N\mathcal L_{\mathrm{sf}}
}
<
\frac12,
\]
then
$\beta_q(K(\theta))
=
\beta_q(K_{\mathrm{tar}})$.
However, the filtered matrix contains more information than its trace.
The trace
$\operatorname{Tr}F$
measures only the total low-frequency spectral mass, whereas
$F=f(L)$
also retains information about the orientation and distribution of the low-frequency eigenspaces in the ambient chain space.
In the sharp-filter and hard-complex limit, \(f(L_q)\) approaches the orthogonal projector onto the harmonic subspace $\ker L_q$.

If \(P\) and \(P_{\mathrm{tar}}\) are orthogonal projectors onto two subspaces of equal dimension, then
\[
\frac12
\left\|
P-P_{\mathrm{tar}}
\right\|_F^2
=
\sum_i
\sin^2\vartheta_i,
\]
where \(\vartheta_i\) are the principal angles between the two subspaces.
Thus, spectral-filter matching can compare not only $\dim\ker L_q
=
\beta_q$,
but also the placement of the low-frequency subspace within the ambient chain space.
At finite filter bandwidth, near-zero positive-eigenvalue modes also contribute. Hence the filtered operator may encode low-frequency geometric and spectral information that is not represented by the Betti number alone.
Accordingly, we distinguish two types of objectives:
$\operatorname{Tr}
f(\widehat L_q^{(\varepsilon)})$
is used when the primary goal is to control the amount of \(q\)-dimensional homological structure, whereas
$f(\widehat L_q^{(\varepsilon)})$
itself is matched to a target filtered operator when the target object is known and its broader low-frequency Hodge spectral structure is to be reproduced.

\subsection{Normalized Polynomial Spectral Moments}

In addition to the heat and resolvent filters, we consider a polynomial low-pass construction designed to control normalized Betti numbers.

One motivation for this formulation comes from recent developments in quantum algorithms for topological data analysis, where normalized Betti numbers are frequently considered instead of unnormalized Betti numbers. For a simplicial complex \(K\), the normalized \(q\)-th Betti number is defined by
\[
\overline\beta_q(K)
=
\frac{
\beta_q(K)
}{
|K^{(q)}|
},
\]
where \(K^{(q)}\) denotes the set of \(q\)-simplices in \(K\).
This quantity measures the fraction of the \(q\)-dimensional chain space occupied by harmonic modes and provides a natural scale-independent topological quantity for comparing complexes of different sizes.

\subsubsection{Penalty-Free Weighted Operator and Fixed Ambient Spectral Scale}

The heat- and resolvent-based objectives considered above use the ambient operator
\[
\widehat L_q^{(\varepsilon)}
=
\left(\widetilde B_q^{(\varepsilon)}\right)^\top
\widetilde B_q^{(\varepsilon)}
+
\widetilde B_{q+1}^{(\varepsilon)}
\left(\widetilde B_{q+1}^{(\varepsilon)}\right)^\top
+
\mu\left(I-W_q^{(\varepsilon)}\right),
\]
where the last term separates inactive simplex directions from the low-frequency spectrum.

For the \(W_q\)-weighted polynomial observable considered in this subsection, however, the inactive-direction penalty is not required. We therefore introduce the penalty-free weighted-current operator
\[
L_{q,\mathrm{wc}}^{(\varepsilon)}
=
\left(\widetilde B_q^{(\varepsilon)}\right)^\top
\widetilde B_q^{(\varepsilon)}
+
\widetilde B_{q+1}^{(\varepsilon)}
\left(\widetilde B_{q+1}^{(\varepsilon)}\right)^\top.
\]
Thus,
\[
\widehat L_q^{(\varepsilon)}
=
L_{q,\mathrm{wc}}^{(\varepsilon)}
+
\mu\left(I-W_q^{(\varepsilon)}\right).
\]
The role of the factor \(W_q^{(\varepsilon)}\) in the observable below is to suppress inactive \(q\)-simplex directions directly, rather than shifting them spectrally by a penalty term. The \(q=1\) specialization of this penalty-free construction is used in the graph experiments of Section~6.

Let $L_q^{\mathrm{amb}}
=
B_q^\top B_q
+
B_{q+1}B_{q+1}^\top$
denote the \(q\)-th Hodge Laplacian of the fixed ambient complex \(K_{\max}\), and define
$\Lambda_q^{\mathrm{amb}}
=
\lambda_{\max}
\left(
L_q^{\mathrm{amb}}
\right)$.
We assume \(\Lambda_q^{\mathrm{amb}}>0\).
Since the ambient complex is fixed throughout optimization,
\(\Lambda_q^{\mathrm{amb}}\) is also fixed.
The penalty-free operator can be written as
\[
L_{q,\mathrm{wc}}^{(\varepsilon)}
=
R_q^{(\varepsilon)}
\left[
B_q^\top W_{q-1}^{(\varepsilon)} B_q
+
B_{q+1}W_{q+1}^{(\varepsilon)}B_{q+1}^\top
\right]
R_q^{(\varepsilon)}.
\]
Since
\[
0\preceq W_p^{(\varepsilon)}\preceq I,
\qquad
0\preceq R_q^{(\varepsilon)}\preceq I,
\]
we have
\[
0
\preceq
B_q^\top W_{q-1}^{(\varepsilon)}B_q
+
B_{q+1}W_{q+1}^{(\varepsilon)}B_{q+1}^\top
\preceq
L_q^{\mathrm{amb}}
\preceq
\Lambda_q^{\mathrm{amb}}I.
\]
Consequently,
\[
0
\preceq
L_{q,\mathrm{wc}}^{(\varepsilon)}
\preceq
\Lambda_q^{\mathrm{amb}}I.
\]

We therefore define the polynomial low-pass filter by
\[
f_{q,d}^{\mathrm{poly}}(\lambda)
=
\left(
1-
\frac{\lambda}
{\Lambda_q^{\mathrm{amb}}}
\right)^d,
\qquad
d\in\mathbb N,
\]
and the corresponding matrix filter by
\[
F_{q,d}^{(\varepsilon)}
=
\left(
I-
\frac{
L_{q,\mathrm{wc}}^{(\varepsilon)}
}{
\Lambda_q^{\mathrm{amb}}
}
\right)^d.
\]
The spectrum of the matrix inside the power lies in \([0,1]\).
At zero eigenvalues,
$f_{q,d}^{\mathrm{poly}}(0)=1$,
whereas the contribution of positive eigenvalues decreases with increasing \(d\).

\subsubsection{Effective Simplex Number and Normalized Spectral Moment}

Define the effective number of \(q\)-simplices in the soft complex by
\[
N_q^{\mathrm{eff}}(\theta)
=
\operatorname{Tr}
W_q^{(\varepsilon)}(\theta)
=
\sum_{\sigma\in K_{\max}^{(q)}}
w_\sigma^{(\varepsilon)}(\theta).
\]
In the hard limit, $(W_q^{(\varepsilon)}
\longrightarrow
\Pi_q)$,
$N_q^{\mathrm{eff}}
\longrightarrow
\operatorname{Tr}\Pi_q
=
|K^{(q)}|$.
Thus, \(N_q^{\mathrm{eff}}\) is a differentiable surrogate for the number of active \(q\)-simplices.

Define the weighted polynomial spectral mass by
\[
S_{q,d}^{W,(\varepsilon)}
=
\operatorname{Tr}
\left[
W_q^{(\varepsilon)}
F_{q,d}^{(\varepsilon)}
\right],
\]
and the normalized polynomial spectral moment by
\[
\overline S_{q,d}^{W,(\varepsilon)}
=
\frac{
\operatorname{Tr}
\left[
W_q^{(\varepsilon)}
\left(
I-
\frac{
L_{q,\mathrm{wc}}^{(\varepsilon)}
}{
\Lambda_q^{\mathrm{amb}}
}
\right)^d
\right]
}{
\operatorname{Tr}
W_q^{(\varepsilon)}
}.
\]
The weighting by \(W_q^{(\varepsilon)}\) is essential in the penalty-free formulation. In the hard limit, the penalty-free operator has zero directions corresponding both to homological modes and to inactive ambient simplices. Multiplication by \(W_q^{(\varepsilon)}\), however, removes the inactive directions from the observable itself
\footnote{
For numerical stability, one may use
\[
\overline S_{q,d,\delta}^{W,(\varepsilon)}
=
\frac{
S_{q,d}^{W,(\varepsilon)}
}{
N_q^{\mathrm{eff}}+\delta
},
\]
with a small constant \(\delta>0\). The theoretical results below are stated for \(\delta=0\).
}.

\begin{proposition}[Hard-limit consistency of the \(W_q\)-weighted penalty-free moment]
Let \(K\) be the corresponding hard complex and define 
$n_q
=
|K^{(q)}|,
~
b
=
\beta_q(K)$.
With respect to the decomposition
\[
C_q^{\mathrm{amb}}
=
C_q(K)
\oplus
C_q(K)^\perp,
\]
the hard limit of the penalty-free operator satisfies
\[
L_{q,\mathrm{wc}}^{\mathrm{hard}}
=
L_q(K)
\oplus
0,
\]
whereas
\[
W_q^{\mathrm{hard}}
=
I_{n_q}
\oplus
0.
\]
Consequently,
\[
\overline S_{q,d}^{W,\mathrm{hard}}
=
\frac{1}{n_q}
\operatorname{Tr}_{C_q(K)}
\left(
I-
\frac{
L_q(K)
}{
\Lambda_q^{\mathrm{amb}}
}
\right)^d.
\]
If \(b<n_q\), let \(\gamma_q>0\) be the smallest positive eigenvalue of \(L_q(K)\). Then
\[
0
\le
\overline S_{q,d}^{W,\mathrm{hard}}
-
\overline\beta_q(K)
\le
\left(
1-
\frac{
\gamma_q
}{
\Lambda_q^{\mathrm{amb}}
}
\right)^d.
\]
If \(b=n_q\), then
$\overline S_{q,d}^{W,\mathrm{hard}}
=
\overline\beta_q(K)
=
1$
for every \(d\).
\end{proposition}

\begin{proof}
In the hard limit,
\[
R_p^{(\varepsilon)}
\longrightarrow
\Pi_p.
\]
The projected boundary operators coincide with the ordinary boundary operators of \(K\) on the active chain spaces and vanish on inactive simplex directions. Hence
\[
L_{q,\mathrm{wc}}^{\mathrm{hard}}
=
L_q(K)\oplus 0.
\]
At the same time,
\[
W_q^{\mathrm{hard}}
=
\Pi_q
=
I_{n_q}\oplus 0.
\]
Therefore,
\[
W_q^{\mathrm{hard}}
\left(
I-
\frac{
L_{q,\mathrm{wc}}^{\mathrm{hard}}
}{
\Lambda_q^{\mathrm{amb}}
}
\right)^d
=
\left(
I-
\frac{
L_q(K)
}{
\Lambda_q^{\mathrm{amb}}
}
\right)^d
\oplus
0,
\]
which proves
\[
\overline S_{q,d}^{W,\mathrm{hard}}
=
\frac{1}{n_q}
\operatorname{Tr}_{C_q(K)}
\left(
I-
\frac{
L_q(K)
}{
\Lambda_q^{\mathrm{amb}}
}
\right)^d.
\]

Now suppose \(b<n_q\), and let the eigenvalues of \(L_q(K)\) be ordered as
\[
0
=
\lambda_1
=
\cdots
=
\lambda_b
<
\lambda_{b+1}
\le
\cdots
\le
\lambda_{n_q}.
\]
Since the zero eigenvalues contribute one to the polynomial filter,
\[
\overline S_{q,d}^{W,\mathrm{hard}}
=
\frac{b}{n_q}
+
\frac{1}{n_q}
\sum_{i>b}
\left(
1-
\frac{
\lambda_i
}{
\Lambda_q^{\mathrm{amb}}
}
\right)^d.
\]
For every positive eigenvalue,
\[
\lambda_i\ge\gamma_q.
\]
Hence,
\[
0
\le
\left(
1-
\frac{
\lambda_i
}{
\Lambda_q^{\mathrm{amb}}
}
\right)^d
\le
\left(
1-
\frac{
\gamma_q
}{
\Lambda_q^{\mathrm{amb}}
}
\right)^d.
\]
It follows that
\[
\begin{aligned}
0
&\le
\overline S_{q,d}^{W,\mathrm{hard}}
-
\frac{b}{n_q}
\\
&\le
\frac{
n_q-b
}{
n_q
}
\left(
1-
\frac{
\gamma_q
}{
\Lambda_q^{\mathrm{amb}}
}
\right)^d
\\
&\le
\left(
1-
\frac{
\gamma_q
}{
\Lambda_q^{\mathrm{amb}}
}
\right)^d.
\end{aligned}
\]
This proves the result.
\end{proof}

\subsubsection{Approximation in the Soft Regime}

We next relate the penalty-free soft operator to its hard-limit counterpart.
At a regular parameter value satisfying
\[
m(\theta)
=
\min_{e\in E_{\max}}
|a_e(\theta)|
>
0,
\]
the same weighted-boundary estimates used for the general ambient operator imply
\[
\left\|
L_{q,\mathrm{wc}}^{(\varepsilon)}
-
L_{q,\mathrm{wc}}^{\mathrm{hard}}
\right\|_2
\le
C_q^{\mathrm{wc}}
e^{-m(\theta)/(2\varepsilon)}
\]
for some constant \(C_q^{\mathrm{wc}}>0\), together with
\[
\left\|
W_q^{(\varepsilon)}
-
\Pi_q
\right\|_2
=
O\left(
e^{-m(\theta)/\varepsilon}
\right).
\]

For fixed \(d\), define
\[
M_q^{(\varepsilon)}
=
I-
\frac{
L_{q,\mathrm{wc}}^{(\varepsilon)}
}{
\Lambda_q^{\mathrm{amb}}
},
\qquad
M_q^{\mathrm{hard}}
=
I-
\frac{
L_{q,\mathrm{wc}}^{\mathrm{hard}}
}{
\Lambda_q^{\mathrm{amb}}
}.
\]
Both matrices have operator norm at most one.
Using the telescoping identity
\[
A^d-B^d
=
\sum_{k=0}^{d-1}
A^{d-1-k}(A-B)B^k,
\]
we obtain
\[
\left\|
\left(M_q^{(\varepsilon)}\right)^d
-
\left(M_q^{\mathrm{hard}}\right)^d
\right\|_2
\le
\frac{
d C_q^{\mathrm{wc}}
}{
\Lambda_q^{\mathrm{amb}}
}
e^{-m(\theta)/(2\varepsilon)}.
\]

If the corresponding hard complex contains at least one \(q\)-simplex, then
\[
\operatorname{Tr}
W_q^{(\varepsilon)}
\longrightarrow
n_q>0,
\]
and hence the denominator remains bounded away from zero for sufficiently small \(\varepsilon\).
Combining the preceding estimates, there exists, for every fixed \(d\), a constant
\(C_{q,d}>0\) such that
\[
\left|
\overline S_{q,d}^{W,(\varepsilon)}
-
\overline S_{q,d}^{W,\mathrm{hard}}
\right|
\le
C_{q,d}
e^{-m(\theta)/(2\varepsilon)}.
\]
Therefore,
\[
\left|
\overline S_{q,d}^{W,(\varepsilon)}
-
\overline\beta_q(K)
\right|
\le
C_{q,d}
e^{-m(\theta)/(2\varepsilon)}
+
\left(
1-
\frac{
\gamma_q
}{
\Lambda_q^{\mathrm{amb}}
}
\right)^d,
\]
with the second term omitted when \(b=n_q\).

The first term represents the soft-to-hard approximation error, whereas the second represents the finite-degree polynomial filtering error.
In particular, the rigorous iterated limit is
\[
\lim_{d\to\infty}
\lim_{\varepsilon\to0}
\overline S_{q,d}^{W,(\varepsilon)}
=
\overline\beta_q(K),
\]
provided that the corresponding hard combinatorial structure remains fixed.
More generally, convergence also holds along any sequence
\[
\varepsilon_k\to0,
\qquad
d_k\to\infty,
\]
for which
\[
C_{q,d_k}
e^{-m(\theta)/(2\varepsilon_k)}
\longrightarrow
0.
\]
Thus, the \(W_q\)-weighted normalized polynomial spectral moment provides a differentiable surrogate for the normalized Betti number of the associated hard complex.

\subsubsection{Control of a Target Normalized Betti Number}

At finite softness and finite polynomial degree, the differentiable spectral target and the desired hard normalized Betti target need not coincide exactly.
We therefore distinguish a prescribed soft spectral target
$\overline s_q^\star$
from a desired hard normalized Betti target
$\overline\beta_q^{\mathrm{tar}}$.
We define
\[
\mathcal L_{\mathrm{norm}}
=
\frac12
\left(
\overline S_{q,d}^{W,(\varepsilon)}
-
\overline s_q^\star
\right)^2.
\]

By the triangle inequality,
\[
\begin{aligned}
\left|
\overline\beta_q(K)
-
\overline\beta_q^{\mathrm{tar}}
\right|
\le{}&
\left|
\overline\beta_q(K)
-
\overline S_{q,d}^{W,(\varepsilon)}
\right|
+
\left|
\overline S_{q,d}^{W,(\varepsilon)}
-
\overline s_q^\star
\right|
+
\left|
\overline s_q^\star
-
\overline\beta_q^{\mathrm{tar}}
\right|\\
\le{}&
C_{q,d}
e^{-m(\theta)/(2\varepsilon)}
+
\left(
1-
\frac{
\gamma_q
}{
\Lambda_q^{\mathrm{amb}}
}
\right)^d
+
\sqrt{
2\mathcal L_{\mathrm{norm}}
}
+
\left|
\overline s_q^\star
-
\overline\beta_q^{\mathrm{tar}}
\right|.
\end{aligned}
\]
Thus, hard normalized-Betti control depends on three effects:
the soft-to-hard approximation error, the finite-degree polynomial filtering error, and the finite-regime mismatch between the chosen soft spectral target and the desired hard topological target.

In the special case
$\overline s_q^\star
=
\overline\beta_q^{\mathrm{tar}}$,
the preceding estimate reduces to
\[
\left|
\overline\beta_q(K)
-
\overline\beta_q^{\mathrm{tar}}
\right|
\le
C_{q,d}
e^{-m(\theta)/(2\varepsilon)}
+
\left(
1-
\frac{
\gamma_q
}{
\Lambda_q^{\mathrm{amb}}
}
\right)^d
+
\sqrt{
2\mathcal L_{\mathrm{norm}}
}.
\]

If the number \(n_q\) of active \(q\)-simplices is fixed and
$\overline\beta_q^{\mathrm{tar}}
=
\frac{
b_{\mathrm{tar}}
}{
n_q
}$,
then the stronger condition
\[
C_{q,d}
e^{-m(\theta)/(2\varepsilon)}
+
\left(
1-
\frac{
\gamma_q
}{
\Lambda_q^{\mathrm{amb}}
}
\right)^d
+
\sqrt{
2\mathcal L_{\mathrm{norm}}
}
+
\left|
\overline s_q^\star
-
\overline\beta_q^{\mathrm{tar}}
\right|
<
\frac{1}{2n_q}
\]
implies
$\beta_q(K)
=
b_{\mathrm{tar}}$.
When \(n_q\) itself varies during optimization, the same estimate still provides quantitative control of $\overline\beta_q(K)$.
Accordingly, the asymptotic theory establishes consistency between the weighted soft spectral moment and the hard normalized Betti number, while at finite softness and finite polynomial degree the soft spectral target may be calibrated separately for the desired hard-topology regime.

\section{Numerical Experiments: Vietoris--Rips Topology Control}
\label{sec:vr_experiments}

In this section, we apply the Hodge spectral surrogate developed in the preceding sections to Vietoris--Rips complexes of point clouds and numerically investigate its properties as a gradient-based signal for topology-constrained optimization.
Persistent homology and Hodge spectral filtering exploit different types of information; therefore, the comparison here concerns their behavior as optimization signals rather than their relative merits as topological descriptors.

The experiments show that increasing the low-frequency Hodge spectral mass can be accompanied by an increase in the hard ($H_1$) structure, while the proposed objectives yield more spatially distributed gradients and smaller scale-normalized derivative variation near persistence-pairing changes in the tested settings. Matrix-valued spectral-filter matching also produces update directions that reflect prescribed geometric deformations, although the current dense implementation incurs substantially higher computational cost.

\subsection{Application to Vietoris--Rips Complexes}

Let $X=\{x_1,\ldots,x_n\}\subset\mathbb R^d$
be a point cloud.
At scale \(r>0\), the Vietoris--Rips complex contains an edge \((i,j)\) whenever
$\|x_i-x_j\|\le r$.
We continuously relax this hard condition by defining
\[
p_{ij}^{(r,\varepsilon)}(X)
=
\sigma\left(
\frac{r-\|x_i-x_j\|}{\varepsilon}
\right).
\]
For a simplex \(\sigma\), define
$w_\sigma^{(r,\varepsilon)}(X)
=
\prod_{e\subset\sigma}
p_e^{(r,\varepsilon)}(X)$.
Since the Vietoris--Rips condition requires all edges of a simplex to have length at most \(r\),
\[
w_\sigma^{(r,\varepsilon)}(X)
\longrightarrow
\mathbf 1_{\{\sigma\in\mathrm{VR}_r(X)\}}
\qquad
(\varepsilon\to0).
\]
Thus, at every scale \(r\), we obtain a differentiable Hodge spectral operator
\[
X
\longmapsto
\widehat L_q^{(r,\varepsilon)}(X).
\]
By the theory developed in the preceding sections, in a regular region where no pairwise distance is exactly equal to the threshold \(r\),
\(\widehat L_q^{(r,\varepsilon)}(X)\) is close to the corresponding hard operator, and its low-frequency spectral mass
\[
S_q^{(r)}(X)
=
\operatorname{Tr}
f\left(
\widehat L_q^{(r,\varepsilon)}(X)
\right)
\]
provides a differentiable surrogate for
$\beta_q\left(\mathrm{VR}_r(X)\right)$.
We use this correspondence as the basis of the numerical experiments below.

\subsection{Multi-Scale Hodge Spectral Objectives}

We consider multiple filtration scales
$r_1<r_2<\cdots<r_M$.
For each scale, define
\[
S_q^{(r_m)}(X)
=
\operatorname{Tr}
f\left(
\widehat L_q^{(r_m,\varepsilon)}(X)
\right).
\]
For scale weights \(\omega_m\ge0\), define the multi-scale spectral mass by
\[
S_q^I(X)
=
\sum_{m=1}^{M}
\omega_m
S_q^{(r_m)}(X),
\]
and the corresponding hard Betti profile by
$B_q^I(X)
=
\sum_{m=1}^{M}
\omega_m
\beta_q\left(
\mathrm{VR}_{r_m}(X)
\right)$.
The preceding theory implies that \(S_q^I(X)\) can be interpreted as a smooth surrogate for \(B_q^I(X)\).
Accordingly, to promote \(q\)-dimensional homological structure, we minimize
$\mathcal L_{\mathrm{promote}}
=
-S_q^I(X)$,
whereas to suppress it, we minimize
$\mathcal L_{\mathrm{suppress}}
=
S_q^I(X)$.
If a target Betti profile \(b_q^{\mathrm{tar}}(r_m)\) is specified, we use
\[
\mathcal L_{\mathrm{trace}}
=
\frac12
\sum_m
\omega_m
\left[
S_q^{(r_m)}(X)
-
b_q^{\mathrm{tar}}(r_m)
\right]^2.
\]
These are not persistence-diagram losses: rather than directly optimizing birth--death pairs, the proposed method optimizes the low-frequency Hodge spectrum at each scale.

When the target point cloud \(X_{\mathrm{tar}}\) itself is known, we compare the full filtered operators rather than only their traces.
We define
\[
\mathcal L_{\mathrm{sf}}
=
\frac12
\sum_m
\omega_m
\left\|
f(L_m(X))
-
f(L_m(X_{\mathrm{tar}}))
\right\|_F^2,
\]
where $L_m(X)
=
\widehat L_q^{(r_m,\varepsilon)}(X)$.
This matrix-valued spectral matching may exploit low-frequency geometric and spectral information that is not represented by the Betti number alone.

\subsection{Experimental Setup}

We primarily consider \(q=1\) and optimize point clouds
$X\subset\mathbb R^2$.
For the Hodge-based methods, we construct
\(\widehat L_1^{(r,\varepsilon)}(X)\)
and apply either a heat or resolvent low-pass filter.
For the persistent-homology baseline, we compute the \(H_1\) barcode and construct objectives from persistence values \(d_\gamma-b_\gamma\). The first four experiments are repeated over 15 random seeds, while the runtime experiment is averaged over five random seeds.

When necessary, mild geometric regularization terms are added to prevent excessive collapse or spreading of the point cloud.
These terms are used only to keep the geometry numerically well behaved; the topological optimization signal itself is supplied by either the persistent-homology or Hodge spectral objective.

\subsection{Multi-Scale \(H_1\) Induction}
\label{sec:exp_h1_induction}

We first examine whether increasing the low-frequency Hodge spectral mass is accompanied by an increase in the actual \(H_1\) structure of the corresponding hard Vietoris--Rips complexes.
For each of 15 random seeds, we initialize \(36\) points randomly in $[-1,1]^2$
and use the scale set
$I
=
\{0.42,0.46,0.50,0.54,0.58,0.62\}$.
We compare a persistent-homology baseline, a Hodge heat objective, and a Hodge resolvent objective, all using the same optimization budget of \(10\) epochs.
For the Hodge objectives, we maximize
\[
S_1^I(X)
=
\sum_{r\in I}
\operatorname{Tr}
f\left(
\widehat L_1^{(r,\varepsilon)}(X)
\right).
\]
Figure~\ref{fig:vr_exp1} shows the initial point cloud and the optimized configurations.

\begin{figure}[t]
    \centering
    \includegraphics[width=\textwidth]{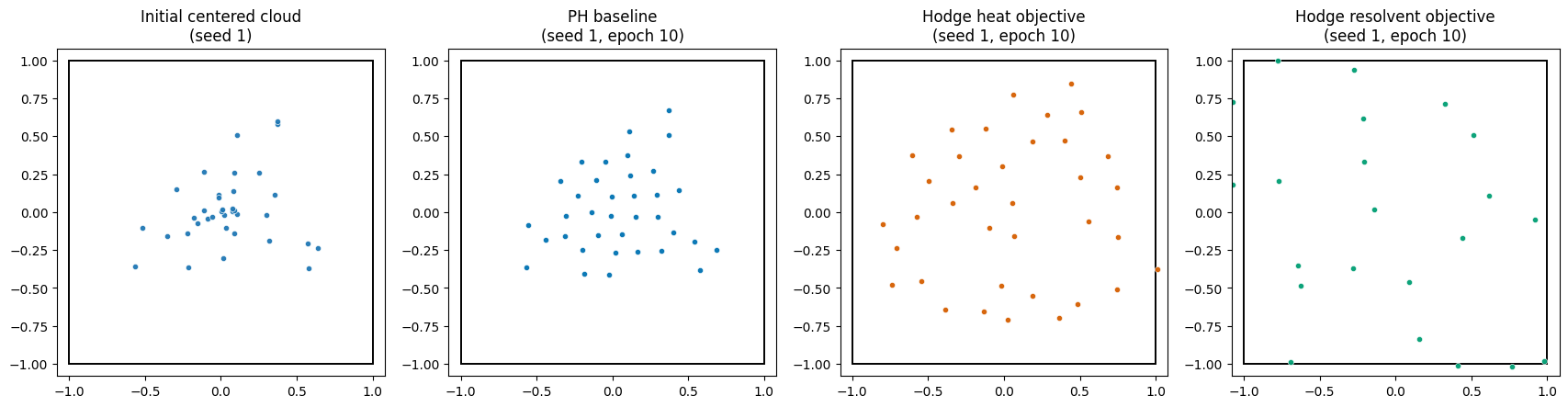}
    \caption{
    Multi-scale \(H_1\) induction from a random point cloud. The panels show one seeded run (seed 1) from the 15 random seeds, and all methods use the same optimization budget of \(10\) epochs. The Hodge heat and resolvent objectives produce more pronounced loop-like structures in this example.
    }
    \label{fig:vr_exp1}
\end{figure}

After optimization, we evaluate the exact hard diagnostic
$B_1^I(X)
=
\sum_{r\in I}
\beta_1\left(
\mathrm{VR}_r(X)
\right)$,
which is not used in the optimization.
The resulting mean \(\pm\) standard deviation values over 15 seeds are
$B_1^I
=
0.400\pm0.712,\;0.400\pm0.879,\;10.067\pm3.473,\;23.600\pm5.352$
for the initial configuration, persistent-homology baseline, Hodge heat objective, and Hodge resolvent objective, respectively.
Thus, across the seeded runs, increasing the soft low-frequency \(H_1\) spectral mass is accompanied by an increase in the actual \(H_1\) structure of the hard Vietoris--Rips complexes.
This does not imply that a larger Betti-profile sum is universally preferable; rather, Figure~\ref{fig:vr_exp1} and the exact Betti diagnostic provide empirical support for the soft-spectral/hard-topological correspondence predicted by the theory.

\subsection{Gradient Localization}
\label{sec:exp_gradient_localization}

We next compare the spatial distribution of the optimization signals.
Figure~\ref{fig:vr_exp2_direction} visualizes the gradient directions generated by the persistent-homology loss and the Hodge heat and resolvent objectives.

\begin{figure}[t]
    \centering
    \includegraphics[width=\textwidth]{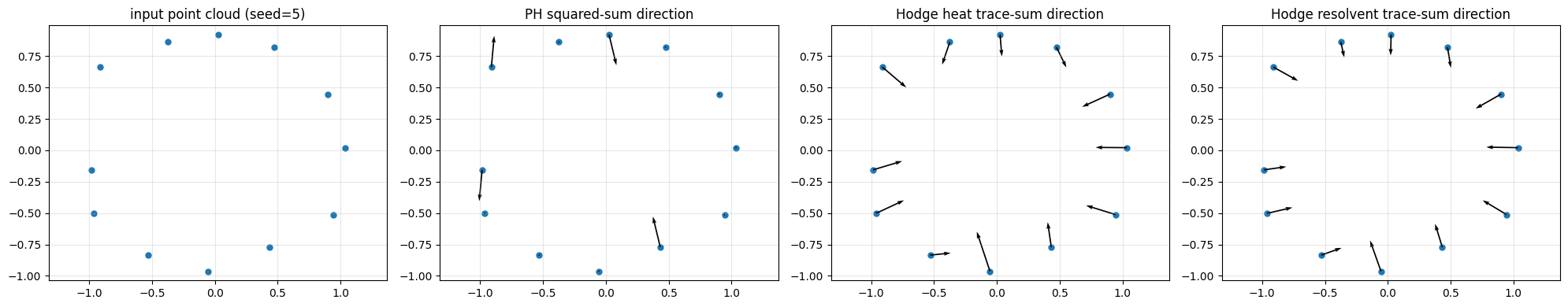}
    \caption{
    Gradient directions produced by the persistent-homology, Hodge heat, and Hodge resolvent objectives. The panels show one seeded run (seed 5) from the 15 random seeds; for this input, the persistent-homology gradient is concentrated on a relatively small subset of points, whereas the Hodge spectral gradients are distributed more broadly over the point cloud.
    }
    \label{fig:vr_exp2_direction}
\end{figure}

For each point, let
$g_i
=
\left\|
\nabla_{x_i}\mathcal L
\right\|$,
and define the normalized gradient-mass distribution
$\pi_i
=
\frac{g_i}{\sum_j g_j}$.
We characterize its spatial spread using the entropy
$H(\pi)
=
-\sum_i
\pi_i\log\pi_i$
and the fraction of total gradient mass carried by the top \(10\%\) of points with the largest gradient norms.
The measured mean \(\pm\) standard deviation values over 15 seeds are
\[
\begin{array}{c|cc}
&
\text{Entropy}
&
\text{Top-10\% mass}
\\
\hline
\mathrm{PH}
&
1.3247\pm0.1231
&
0.5429\pm0.0858
\\
\mathrm{Hodge\ heat}
&
2.4548\pm0.0181
&
0.2211\pm0.0170
\\
\mathrm{Hodge\ resolvent}
&
2.4665\pm0.0137
&
0.2129\pm0.0191
\end{array}
\]
and are summarized in Figure~\ref{fig:vr_exp2_entropy}.

\begin{figure}[t]
    \centering
    \includegraphics[width=0.72\textwidth]{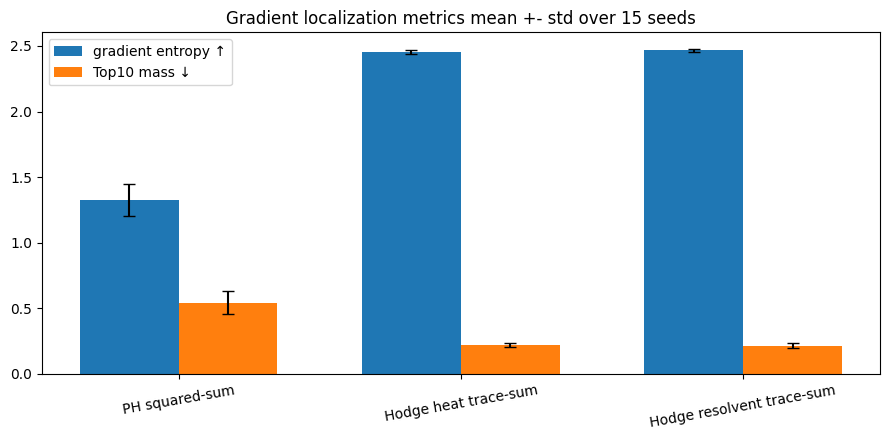}
    \caption{
    Quantitative comparison of gradient localization over 15 random seeds.
    Higher entropy and lower top-\(10\%\) mass indicate a more spatially distributed gradient profile.
    Both Hodge spectral objectives produce substantially more distributed gradients than the persistent-homology loss across the tested seeds.
    }
    \label{fig:vr_exp2_entropy}
\end{figure}

Across the 15 seeded point clouds, the Hodge spectral objectives exhibit higher entropy and lower top-\(10\%\) mass than the persistent-homology loss, indicating a more spatially distributed gradient profile.
This is consistent with the fact that persistence-based losses propagate gradients through critical simplices associated with birth and death events, whereas Hodge spectral losses depend on the low-frequency structure of the ambient matrix operator.
This experiment concerns gradient localization and should not be interpreted as a general performance comparison between the two approaches.

\subsection{Pairing-Switch Stress Test}
\label{sec:exp_pair_switch}

We next examine the behavior of the objectives near a change in persistence pairing. For each of 15 random seeds, we construct a point cloud containing two similar loop-like structures and vary their relative sizes using a deformation parameter \(s\). Configurations from one seeded run are shown in Figure~\ref{fig:vr_exp3_data}.

\begin{figure}[t]
    \centering
    \includegraphics[width=\textwidth]{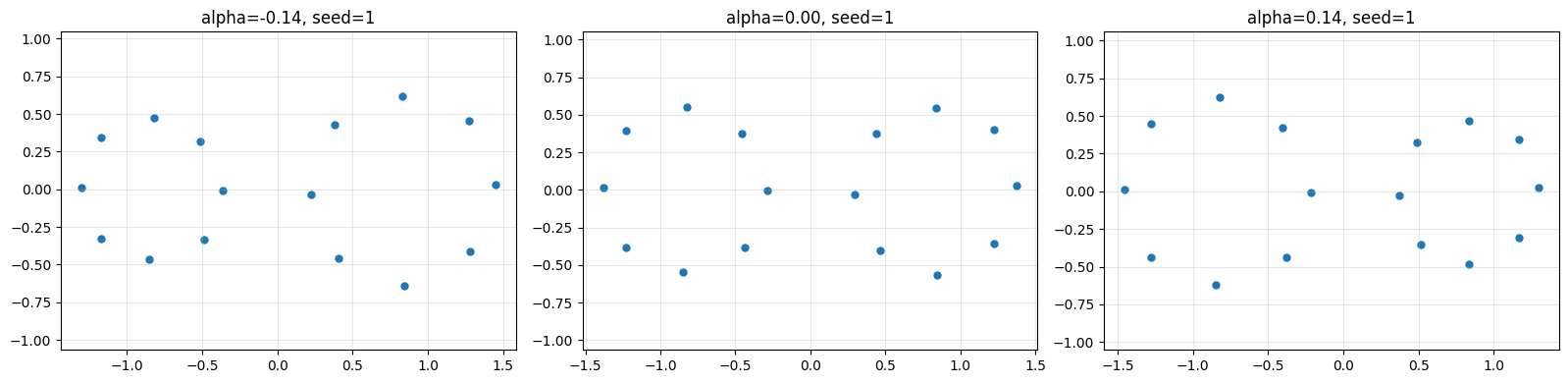}
    \caption{
    Configurations from one seeded run (seed 1) in the pairing-switch stress test for \(s=-0.14\), \(s=0\), and \(s=0.14\). The parameter \(s\) changes the relative sizes of two competing loop-like structures.
    }
    \label{fig:vr_exp3_data}
\end{figure}

For the persistent-homology objective, we use the largest \(H_1\) persistence value.
Near the transition, the ordering of the two dominant persistence bars can change, causing the selected birth--death pair to switch.
The Hodge objective does not explicitly select a particular persistence pair, but instead evaluates low-frequency spectral structure over a scale interval.
Figure~\ref{fig:vr_exp3_loss} shows the objective values and their derivatives with respect to \(s\).

\begin{figure}[t]
    \centering
    \includegraphics[width=\textwidth]{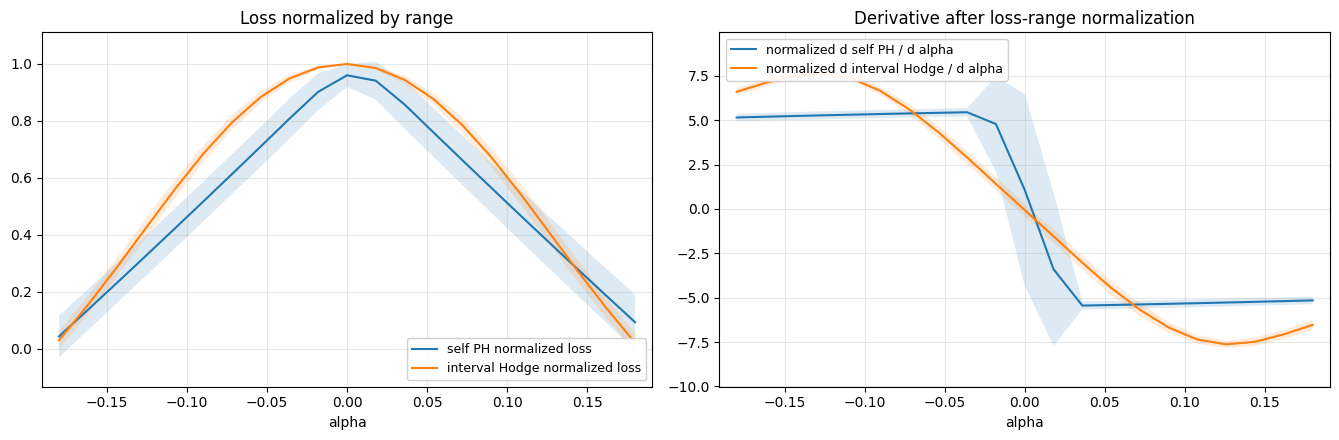}
    \caption{
    Loss values and derivatives in the pairing-switch stress test over 15 random seeds.
    The derivative variation is evaluated both on the raw scale and after normalization by the dynamic range of each loss.
    }
    \label{fig:vr_exp3_loss}
\end{figure}

Because the dynamic ranges of the two losses differ substantially, we normalize each derivative variation by the range of the corresponding objective.
The maximum normalized derivative jumps are then \(10.9800\pm0.4616\) for persistent homology and \(1.5405\pm0.0682\) for the Hodge spectral objective, reported as mean \(\pm\) standard deviation over 15 seeds.
Thus, relative to its own loss scale, the Hodge objective exhibits substantially smaller derivative variation near the pairing switch.

\subsection{Target Geometric Deformation}
\label{sec:exp_target_alignment}

We next investigate whether matrix-valued spectral-filter matching provides optimization information beyond the amount of homology alone.

The initial point cloud is a circle and the target is a wavy ring. The experiment is repeated over 15 random seeds.
Each point is parameterized as
$x_i(r_i)
=
\left(
r_i\cos\theta_i,
r_i\sin\theta_i
\right)$,
where the angles \(\theta_i\) are fixed and the radii \(r_i\) are optimized.
For the Hodge objectives, we minimize
\[
\mathcal L_{\mathrm{sf}}
=
\frac12
\left\|
f(\widehat L_1(X))
-
f(\widehat L_1(X_{\mathrm{tar}}))
\right\|_F^2.
\]
Figure~\ref{fig:vr_exp4_profile} compares the reconstructed radial profiles with the target.

\begin{figure}[t]
    \centering
    \includegraphics[width=\textwidth]{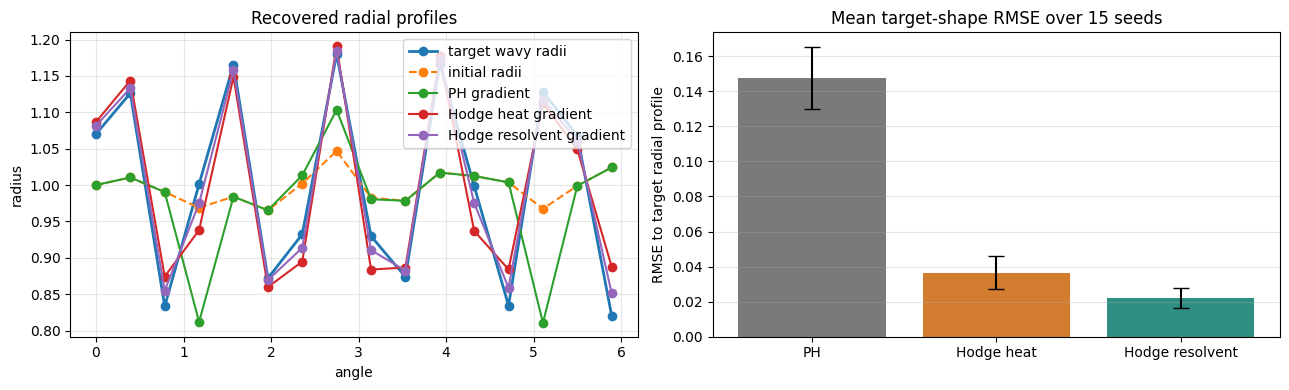}
    \caption{
    Recovered radial profiles in the target-shape experiment. The left panel shows one seeded run (seed 7), while the right panel reports mean \(\pm\) standard deviation RMSE over 15 random seeds; the Hodge heat and resolvent objectives reproduce the target radial modulation more accurately than the persistent-homology baseline in this controlled setting.
    }
    \label{fig:vr_exp4_profile}
\end{figure}

The final RMSE values are
$0.1477\pm0.0177$
for persistent homology,
$0.0366\pm0.0095$
for the Hodge heat filter, and
$0.0221\pm0.0055$
for the Hodge resolvent filter.
The corresponding optimized point clouds are shown in Figure~\ref{fig:vr_exp4_data}.

\begin{figure}[t]
    \centering
    \includegraphics[width=\textwidth]{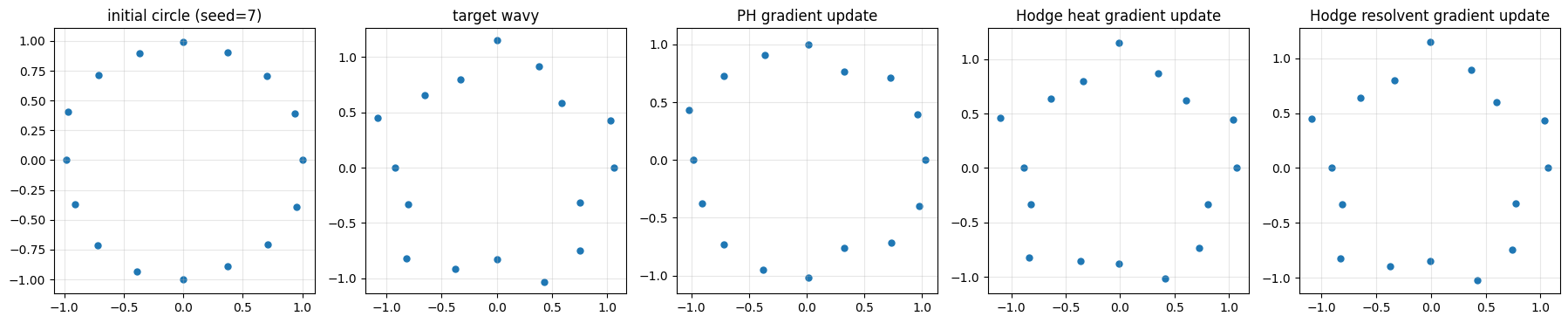}
    \caption{
    Shape synthesis using analytic gradients. Shown are the initial circle from one seeded run (seed 7), the target wavy ring, and the point clouds obtained using the persistent-homology, Hodge heat, and Hodge resolvent objectives.
    }
    \label{fig:vr_exp4_data}
\end{figure}

To examine the local gradient information directly, we also compute the cosine similarity between the first update direction and the prescribed deformation direction toward the target.
The mean \(\pm\) standard deviation cosine similarities over 15 seeds are
\[
-0.1743\pm0.2820,\qquad
0.2818\pm0.1580,\qquad
0.7843\pm0.0734
\]
for persistent homology, the Hodge heat filter, and the Hodge resolvent filter, respectively.
In particular, the resolvent objective produces an initial gradient direction strongly aligned with the target deformation.
Together, Figures~\ref{fig:vr_exp4_profile} and~\ref{fig:vr_exp4_data} support the interpretation that matching the full filtered operator \(f(L)\) can provide low-frequency geometric information that is not represented by the Betti number alone.
This experiment is deliberately controlled and does not establish a general advantage for arbitrary shape-reconstruction tasks.

\subsection{Computational Cost}
\label{sec:exp_runtime}

Finally, we examine the computational cost of the current dense implementation.
We vary the number of points over
$n
=
12,20,24,28,32,36,40,48,56,64,70,80,90,100$
and measure the time required for one loss-gradient evaluation, averaging over five random seeds at each point-cloud size.
Figure~\ref{fig:vr_exp5_runtime} shows the runtime as a function of point-cloud size.

\begin{figure}[t]
    \centering
    \includegraphics[width=\textwidth]{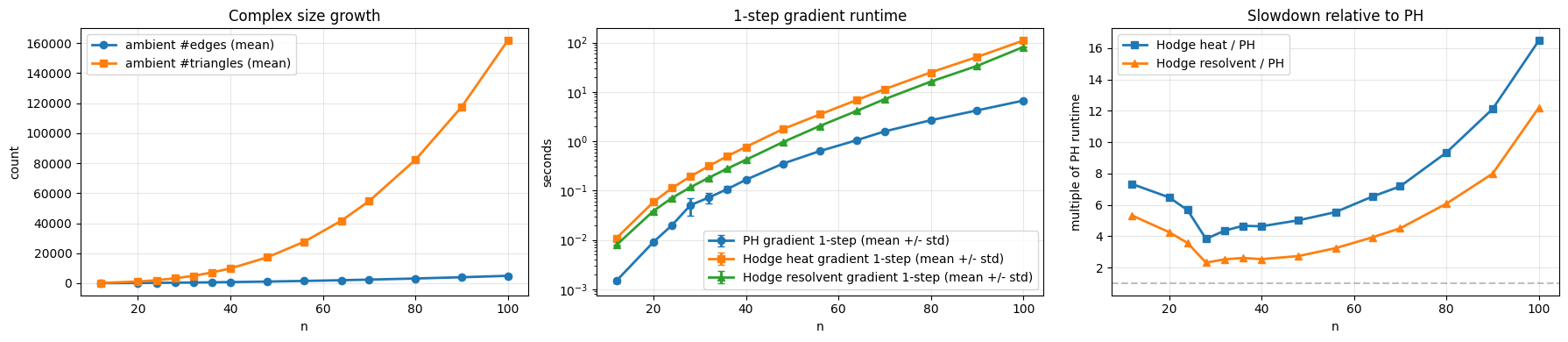}
    \caption{
    Runtime of one loss-gradient evaluation as a function of the number of points, reported as mean \(\pm\) standard deviation over five random seeds.
    The current Hodge implementation uses dense ambient Hodge matrices and is therefore substantially more expensive than the persistent-homology baseline at larger problem sizes.
    }
    \label{fig:vr_exp5_runtime}
\end{figure}

At \(n=100\), the measured runtimes are
$6.6771\pm0.1530\ {\rm s}$
for persistent homology,
$110.1246\pm0.3557\ {\rm s}$
for Hodge heat, and
$81.5280\pm0.6970\ {\rm s}$
for Hodge resolvent.
Thus, relative to the persistent-homology baseline, the current implementation is approximately \(16.49\times\) slower for the heat objective and \(12.21\times\) slower for the resolvent objective.
These results reflect the cost of the present dense proof-of-concept implementation rather than the intrinsic computational complexity of the framework.
Larger-scale applications will require more efficient implementations based on sparse linear algebra, polynomial filtering, Chebyshev approximation, iterative linear solvers, or stochastic trace estimation.

\section{Normalized First-Betti-Type Control on Graph Clique Complexes via Polynomial Moments}
\label{sec:graph_experiments}

In this section, we examine whether the \(W_q\)-weighted normalized polynomial spectral moment introduced in the theoretical section can serve as a differentiable objective for controlling first-homology-type structure in graph clique complexes.
The experiments are intended as controlled proof-of-concept studies on small graphs rather than as large-scale graph-generation benchmarks.

\subsection{Numerical Setup}

For each candidate edge \(e=(i,j)\), we introduce an edge logit \(a_e\) and define the soft edge probability
\[
p_e
=
\sigma(a_e)
=
\frac{1}{1+\exp(-a_e)}.
\]
For a triangle \(\sigma\), the corresponding soft simplex weight is
$w_\sigma
=
\prod_{e\subset\sigma}p_e$.
We define
$W_1
=
\operatorname{diag}(p_e),
~
W_2
=
\operatorname{diag}(w_\sigma)$.
Theoretically,
$R_1=W_1^{1/2},
~
R_2=W_2^{1/2}$,
whereas in the numerical implementation we introduce the stabilization parameter
$\epsilon_w=10^{-8}$
and use
\[
R_1
=
\operatorname{diag}
\left(
\sqrt{p_e+\epsilon_w}
\right),
\qquad
R_2
=
\operatorname{diag}
\left(
\sqrt{w_\sigma+\epsilon_w}
\right).
\]

The weighted boundary operators are
\[
\widetilde B_1
=
B_1R_1,
\qquad
\widetilde B_2
=
R_1B_2R_2,
\]
and the soft first Hodge-Laplacian-type operator is
\[
L_1^{\mathrm{soft}}
=
\widetilde B_1^\top\widetilde B_1
+
\widetilde B_2\widetilde B_2^\top.
\]

No inactive-direction penalty is used in the present implementation.
Instead, inactive directions in the fixed ambient edge space are suppressed directly by the \(W_1\)-weighted trace.
This corresponds to the penalty-free normalized polynomial construction developed in the theoretical section.

For normalization of the polynomial filter, we use the fixed ambient Hodge Laplacian
\[
L_1^{\mathrm{amb}}
=
B_1^\top B_1
+
B_2B_2^\top
\]
and define
\[
\Lambda_1^{\mathrm{amb}}
=
\lambda_{\max}
\left(
L_1^{\mathrm{amb}}
\right).
\]
For the \(15\)-node candidate clique complex considered here,
\[
n=15,
\qquad
\Lambda_1^{\mathrm{amb}}=15,
\qquad
\alpha_1=\frac{1}{15}.
\]
Let
\[
M
=
I-
\frac{L_1^{\mathrm{soft}}}
{\Lambda_1^{\mathrm{amb}}}.
\]
The \(W_1\)-weighted degree-\(d\) polynomial spectral mass is
\[
S_{1,d}^{W}
=
\operatorname{Tr}
\left(
W_1M^d
\right)
=
\sum_e p_e(M^d)_{ee}.
\]
The effective soft number of edges is
\[
N_1^{\mathrm{eff}}
=
\operatorname{Tr}(W_1)
=
\sum_ep_e.
\]
We therefore optimize the normalized moment
\[
\overline S_{1,d}^{W}
=
\frac{
\operatorname{Tr}
\left(
W_1M^d
\right)
}{
\operatorname{Tr}(W_1)+\delta
},
\qquad
\delta=10^{-6}.
\]
In the hard limit, \(W_1\) approaches the indicator of the active edge directions, so inactive ambient directions do not contribute artificial zero modes to the polynomial trace.

Given a prescribed soft spectral target \(\bar{s}_1^\star\), the topological loss is
\[
\mathcal L_{\mathrm{top}}
=
\frac12
\left(
\overline S_{1,d}^{W}
-
\bar{s}_1^\star
\right)^2.
\]

The implementation uses the exact analytical backward computation derived in the Appendix.
In particular, defining
\[
K_W
=
\frac1d
\sum_{k=0}^{d-1}
M^{d-1-k}W_1M^k,
\]
the directional derivative of the weighted numerator is
\[
\partial_u S_{1,d}^{W}
=
\operatorname{Tr}
\left[
(\partial_uW_1)M^d
\right]
-
d\alpha_1
\operatorname{Tr}
\left[
K_W
(\partial_uL_1^{\mathrm{soft}})
\right].
\]
Thus, the backward computation includes both the direct derivative of the \(W_1\)-weighted trace and the derivative propagated through the soft Hodge operator.

\subsection{Soft and Hard Targets}

The soft spectral target \(\bar{s}_1^\star\) and the hard normalized Betti target \(\bar{\beta}_1^\star\) are treated as distinct quantities.
Optimization is performed on \(\overline S_{1,d}^{W}\), whereas hard topology is evaluated by sampling \(128\) Bernoulli graphs from the optimized edge probabilities and computing their normalized first Betti numbers.
\footnote{
Bernoulli sampling provides a natural finite-soft interpretation of the optimized edge variables.
Under independent Bernoulli edge sampling with probabilities \(p_e\), the simplex weight
\[
w_\sigma=\prod_{e\subset\sigma}p_e
\]
is exactly the probability that all edges required for the clique simplex \(\sigma\) are present.
Thus, Bernoulli sampling preserves the probabilistic information carried by intermediate values \(p_e\in(0,1)\), whereas deterministic thresholding would collapse this information to a single binary graph.
This finite-soft probabilistic interpretation is complementary to the deterministic hard-limit analysis developed in the preceding sections.
}

At finite sigmoid softness and finite polynomial degree, we do not assume a universal one-to-one relation
\[
\bar{s}_1^\star
\longleftrightarrow
\bar{\beta}_1^\star.
\]
Instead, for each desired hard target, we first identify an Erd\H{o}s--R\'enyi-type reference probability \(p^\star\) whose Monte Carlo hard normalized first Betti number lies near the desired regime.
The corresponding soft target is then defined by
\[
\bar{s}_1^\star
=
\overline S_{1,d}^{W}
\bigl(
p_e=p^\star
\bigr).
\]
Thus, this calibration represents an experimental finite-problem correspondence rather than a universal analytic conversion rule.

The main experiments use \(24\) random initializations.
The initialization-noise scales are
\[
0.20,\qquad
0.35,\qquad
0.50,
\]
with eight runs at each level.
Within each comparison, the same initial graphs are reused across targets or optimization methods.

\subsection{Soft-Moment Optimization}

We first examine whether the weighted normalized soft spectral moment can be optimized consistently from different initial conditions.

The hard normalized Betti target is
$\bar{\beta}_1^\star=0.10$,
and we use \(d=8\).
The Erd\H{o}s--R\'enyi calibration gives
\[
p^\star=0.42,
\qquad
\bar{s}_1^\star=0.67471964.
\]
The edge logits are optimized using Adam with learning rate \(0.02\).
Because the numerical scale of the \(W_1\)-weighted objective differs from that of the unweighted polynomial trace, convergence is determined from the soft-target error:
\[
\left|
\overline S_{1,8}^{W}
-
\bar{s}_1^\star
\right|
<
10^{-4},
\]
with a maximum of \(160\) iterations.
Figure~\ref{fig:loss_convergence_weighted} shows the optimization trajectories over the \(24\) initializations, and Table~\ref{tab:weighted_betti_optimization} summarizes the aggregate results.

\begin{figure}[t]
    \centering
    \includegraphics[width=0.47\linewidth]{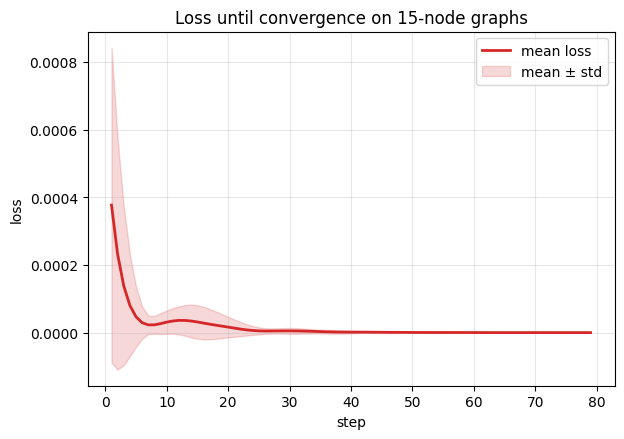}
    \hfill
    \includegraphics[width=0.47\linewidth]{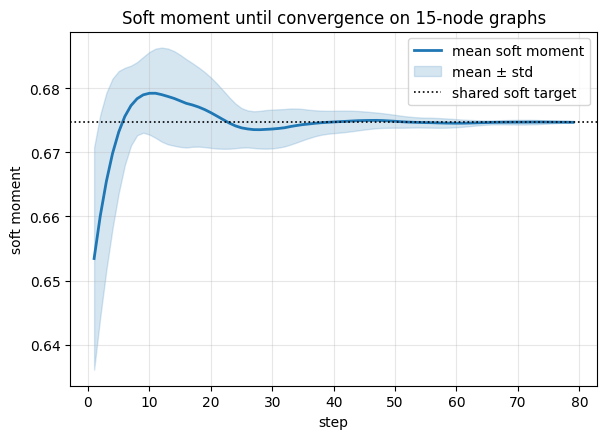}
    \caption{
    Optimization of the \(W_1\)-weighted normalized polynomial spectral moment over \(24\) random initializations on \(15\)-node graphs.
    Left: mean topological loss during optimization.
    Right: mean normalized weighted spectral moment.
    Shaded regions indicate one standard deviation across initializations, and the dotted line denotes the prescribed soft target.
    }
    \label{fig:loss_convergence_weighted}
\end{figure}

\begin{table}[t]
    \centering
    \caption{
    Aggregate results of the \(W_1\)-weighted soft-moment optimization over \(24\) initializations.
    Reported uncertainties are standard deviations across initializations.
    }
    \label{tab:weighted_betti_optimization}
    \begin{tabular}{l|c}
        \hline
        Quantity & Mean \(\pm\) std \\
        \hline
        Final loss
        & \(4.27\times10^{-7}\pm6.73\times10^{-7}\) \\
        Final weighted normalized moment
        & \(0.674839\pm0.000917\) \\
        Absolute soft-target error
        & \(0.000726\pm0.000572\) \\
        Absolute hard-target error
        & \(0.00707\pm0.00482\) \\
        \hline
    \end{tabular}
\end{table}

As shown in Figure~\ref{fig:loss_convergence_weighted}, the weighted soft spectral moment is driven close to the prescribed target across the tested initializations.
The mean hard-target errors at initialization-noise levels \(0.20\), \(0.35\), and \(0.50\) are
\[
0.00437\pm0.00427,
\qquad
0.00525\pm0.00291,
\qquad
0.01160\pm0.00343,
\]
respectively.
Although the hard-topology error tends to increase under larger initialization perturbations, convergence of the soft objective remains stable throughout the tested initialization family.
Thus, the observed optimization behavior is not tied to a single favorable initial graph.

\subsection{Topology Control Across Multiple Targets}

We next investigate whether the sampled hard topology changes systematically across different target regimes.
The hard normalized Betti targets are
$\bar{\beta}_1^\star
\in
\{0.02,0.05,0.10,0.20\}$,
and we use \(d=5\).
The calibrated reference probabilities and corresponding soft targets are listed in Table~\ref{tab:weighted_target_calibration}.

\begin{table}[t]
    \centering
    \caption{
    Calibration values used in the multi-target experiment.
    The soft target is obtained by evaluating the \(W_1\)-weighted normalized polynomial moment at the calibrated uniform reference probability \(p^\star\).
    }
    \label{tab:weighted_target_calibration}
    \begin{tabular}{c|c|c}
        \hline
        Hard target
        & Reference \(p^\star\)
        & Soft target \(\bar{s}_1^\star\) \\
        \hline
        0.02 & 0.53    & 0.577670 \\
        0.05 & 0.10625 & 0.942150 \\
        0.10 & 0.41    & 0.760477 \\
        0.20 & 0.21    & 0.899299 \\
        \hline
    \end{tabular}
\end{table}

The soft targets are not monotone in the hard target.
This is expected because the normalized first Betti number itself is not monotone in graph density, and the finite-soft, finite-degree spectral surrogate depends on the spectral structure of each target regime rather than only on edge density.
For all \(24\) initializations, the same initial soft graph is reused across the four targets.
The mean normalized first Betti number of the initial graphs is
$0.17491\pm0.00899$,
so the experiment includes both decreases and increases relative to the initial topology.
Figure~\ref{fig:weighted_multitarget} compares the prescribed hard targets with the final sampled normalized first Betti numbers.
The numerical results are summarized in Table~\ref{tab:weighted_topology_control}.

\begin{figure}[t]
    \centering
    \includegraphics[width=0.72\linewidth]{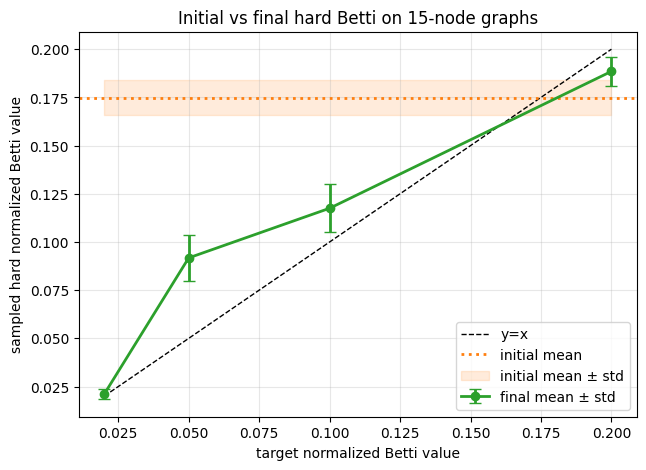}
    \caption{
    Hard-topology control across four target regimes using the \(W_1\)-weighted normalized polynomial spectral objective.
    Each point aggregates \(24\) optimization runs.
    The dashed line indicates \(y=x\).
    The horizontal reference indicates the mean topology of the shared initial graphs, while the final markers show the mean and standard deviation of the sampled hard normalized first Betti numbers.
    }
    \label{fig:weighted_multitarget}
\end{figure}

\begin{table}[t]
    \centering
    \caption{
    Hard-topology results for the multi-target \(W_1\)-weighted spectral-moment experiment.
    Each entry aggregates \(24\) optimization runs.
    }
    \label{tab:weighted_topology_control}
    \begin{tabular}{c|c|c}
        \hline
        Hard target
        & Final normalized \(\beta_1\)
        & Mean absolute hard error \\
        \hline
        0.02 & \(0.02102\pm0.00258\) & \(0.00223\pm0.00164\) \\
        0.05 & \(0.09175\pm0.01179\) & \(0.04175\pm0.01179\) \\
        0.10 & \(0.11750\pm0.01242\) & \(0.01773\pm0.01208\) \\
        0.20 & \(0.18851\pm0.00771\) & \(0.01200\pm0.00690\) \\
        \hline
    \end{tabular}
\end{table}

The corresponding mean absolute soft-target errors are approximately
\[
0.000462,\qquad
0.014273,\qquad
0.000706,\qquad
0.000309
\]
for hard targets \(0.02\), \(0.05\), \(0.10\), and \(0.20\), respectively.
As shown in Figure~\ref{fig:weighted_multitarget}, the final mean hard normalized first Betti numbers satisfy
\[
0.0210
<
0.0918
<
0.1175
<
0.1885,
\]
and therefore reproduce the ordering of the prescribed hard target regimes.
This indicates that the weighted spectral objective provides a meaningful directional signal for changing the hard topology.

Target-matching accuracy, however, is not uniform across regimes.
For \(\bar{\beta}_1^\star=0.05\), the soft objective itself does not fully reach its calibrated target within the specified optimization budget, and the final hard topology remains near \(0.092\).
For the \(0.10\) target, the final hard mean is also somewhat higher than the nominal target.
The multi-target experiment should therefore not be interpreted as demonstrating an exact finite-soft conversion between arbitrary hard Betti targets and calibrated spectral targets.
Rather, the results show that the spectral surrogate is directly differentiable and controllable, and that systematic changes in the surrogate are accompanied by systematic changes in the sampled hard topology.

\subsection{Joint Topology and Graph-Feature Control}

Finally, we examine whether the \(W_1\)-weighted topological objective can be combined with a conventional graph-feature objective.

As the auxiliary graph statistic, we use the soft degree variance.
For vertex \(i\), define
$d_i
=
\sum_{j\neq i}p_{ij},
~
\bar d
=
\frac1n
\sum_i d_i$,
and
\[
V(p)
=
\frac1n
\sum_i
(d_i-\bar d)^2.
\]
We fix the hard normalized Betti target at
$\bar{\beta}_1^\star=0.10$
and consider degree-variance targets
$V^\star
\in
\{0.41,0.45,0.49\}$.
For this experiment,
$d=5,
~
p^\star=0.41$,
so the topology soft target is
$\bar{s}_1^\star
=
0.760477$.

The joint objective is
\[
\mathcal L_{\mathrm{joint}}
=
\frac12
\left(
\overline S_{1,5}^{W}
-
\bar{s}_1^\star
\right)^2
+
\eta
\frac12
\left(
V(p)-V^\star
\right)^2,
\qquad
\eta=0.04.
\]
The loss weight is rebalanced relative to the earlier unweighted formulation because the numerical scale of the \(W_1\)-weighted normalized spectral objective is different.
For each variance target, the variance-only and joint optimizations are initialized from the same \(24\) soft graphs.
Figure~\ref{fig:weighted_joint_control} compares the resulting hard topology and achieved degree variance, and Table~\ref{tab:weighted_joint_objective} summarizes the aggregate numerical results.

\begin{figure}[t]
    \centering
    \includegraphics[width=0.47\linewidth]{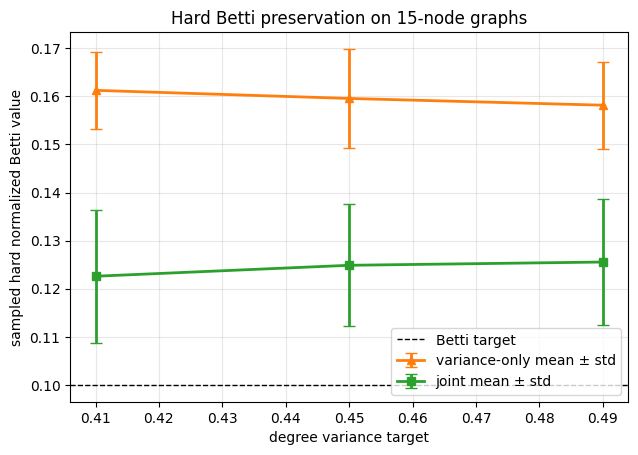}
    \hfill
    \includegraphics[width=0.47\linewidth]{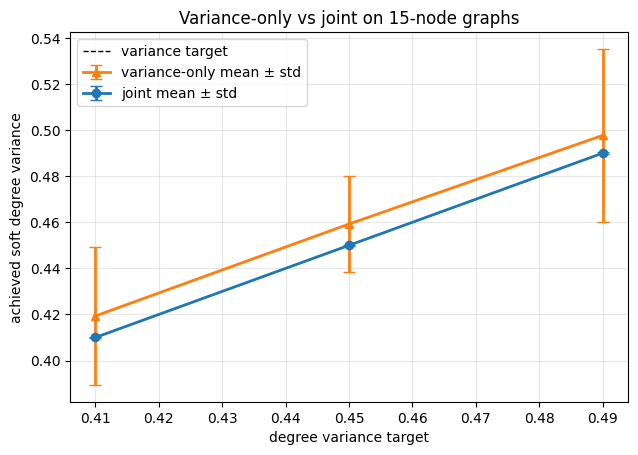}
    \caption{
    Joint topology and graph-feature control over \(24\) random initializations.
    Left: sampled hard normalized first Betti number under variance-only and joint optimization.
    Right: achieved soft degree variance.
    Error bars indicate one standard deviation across initializations.
    }
    \label{fig:weighted_joint_control}
\end{figure}

\begin{table}[t]
    \centering
    \caption{
    Comparison of variance-only and joint optimization over \(24\) initializations.
    }
    \label{tab:weighted_joint_objective}
    \small
    \begin{tabular}{c|c|c|c|c}
        \hline
        Var. target
        & Method
        & Soft variance
        & Hard normalized \(\beta_1\)
        & Mean abs. hard error \\
        \hline
        0.41
        & Variance-only
        & \(0.41934\pm0.02991\)
        & \(0.16124\pm0.00795\)
        & 0.06124 \\
        0.41
        & Joint
        & \(0.410013\pm0.000222\)
        & \(0.12259\pm0.01383\)
        & 0.02259 \\
        \hline
        0.45
        & Variance-only
        & \(0.45924\pm0.02062\)
        & \(0.15955\pm0.01030\)
        & 0.05955 \\
        0.45
        & Joint
        & \(0.449982\pm0.000255\)
        & \(0.12487\pm0.01271\)
        & 0.02487 \\
        \hline
        0.49
        & Variance-only
        & \(0.49768\pm0.03749\)
        & \(0.15815\pm0.00902\)
        & 0.05815 \\
        0.49
        & Joint
        & \(0.489977\pm0.000588\)
        & \(0.12555\pm0.01303\)
        & 0.02569 \\
        \hline
    \end{tabular}
\end{table}

As shown in the right panel of Figure~\ref{fig:weighted_joint_control}, the joint optimization reaches all three degree-variance targets with high accuracy.
At the same time, the topological term systematically shifts the sampled hard normalized first Betti number toward the desired regime.
Under variance-only optimization, the hard normalized first Betti number remains near $0.16$,
whereas with the joint objective it is reduced to approximately
$0.123\text{--}0.126$.
Thus, although the hard topology does not exactly reach the nominal target \(0.10\), the joint objective substantially reduces the hard-topology error while preserving highly accurate degree-variance control.
This indicates that the \(W_1\)-weighted spectral-moment term contains topological information that is not recovered by the degree-variance objective alone.
The remaining deviation from the hard target also indicates that joint topology control is sensitive to objective weighting, calibration, and initialization.

\subsection{Discussion and Limitations}

The \(H_1\) experiments provide three main observations.
First, the theoretically motivated \(W_1\)-weighted normalized polynomial spectral moment is analytically differentiable with respect to the edge logits and can be optimized accurately toward a prescribed soft target across the tested initial conditions, as illustrated in Figure~\ref{fig:loss_convergence_weighted}.
Second, changes in the soft spectral surrogate are accompanied by systematic changes in the sampled hard normalized first Betti number.
As shown in Figure~\ref{fig:weighted_multitarget}, the ordering of the target regimes is reflected in the ordering of the final hard topology.
However, agreement with the target is not uniform across all regimes, and the finite-soft calibrated relation between a soft spectral target and a desired hard Betti target should be distinguished from the asymptotic consistency established in the theoretical analysis.
Third, the weighted spectral objective can be combined with a conventional graph-statistical objective.
As demonstrated in Figure~\ref{fig:weighted_joint_control}, adding the topological term shifts the hard topology substantially toward the desired regime while maintaining highly accurate degree-variance control.

The use of \(24\) initializations across three perturbation levels does not constitute an exhaustive robustness study, but it reduces the possibility that the observed behavior is specific to a single favorable initialization.
The increasing hard-topology error observed for some larger initialization perturbations also indicates that optimization conditioning and finite-regime calibration remain important practical considerations.
The mathematical construction itself extends directly to arbitrary homological degree \(q\).
However, higher-dimensional optimization may require additional care because the chain-space dimensions grow combinatorially and because the simplex weights
$w_\sigma
=
\prod_{e\subset\sigma}p_e$
contain an increasing number of probability factors, which can weaken both the simplex weights and their gradients.
We therefore leave systematic higher-dimensional numerical optimization for future work.

The present results should consequently be viewed as proof-of-concept demonstrations on small graphs.
The soft spectral targets are calibrated separately for each hard-topology regime rather than obtained from a universal finite-soft conversion rule, and the dependence on polynomial degree, optimization budget, initialization, and loss weighting has not yet been studied exhaustively.
Nevertheless, the experiments demonstrate that the \(W_q\)-weighted normalized polynomial spectral moment derived in the theoretical analysis provides a usable differentiable topological signal in graph optimization.
The results support the central mechanism of the proposed framework: low-frequency Hodge spectral information can be converted into a continuous optimization signal whose variation is reflected in the topology of the resulting hard graph distribution.

\section{Summary and Discussion}

In this work, we proposed a framework for incorporating homological constraints into differentiable optimization through the low-frequency spectrum of Hodge-Laplacian-type spectral relaxations.
The central idea is to replace the discontinuous counting of exact topological zero modes by smooth spectral quantities that can be differentiated with respect to the underlying geometric or graph parameters.
Theoretically, the soft Hodge operator approaches the Hodge Laplacian of the corresponding hard complex, and eigenvalue perturbation bounds show that hard topological zero modes persist as near-zero spectral modes in the nearby soft regime.
Accordingly, the proposed spectral quantities should be interpreted as differentiable surrogates for homological structure rather than as new exact topological invariants.

For graph clique complexes, we used the \(W_q\)-weighted normalized polynomial spectral moment
\[
\overline{S}_{q,d}
=
\frac{
\operatorname{Tr}
\left[
W_q
\left(
I-
\frac{L_q^{\mathrm{soft}}}
{\Lambda_q^{\mathrm{amb}}}
\right)^d
\right]
}{
\operatorname{Tr}(W_q)+\delta
},
\]
where \(\Lambda_q^{\mathrm{amb}}\) is a fixed spectral scale determined by the ambient clique complex.
The \(W_q\) factor suppresses inactive simplex directions in the fixed ambient representation and yields a quantity that is consistent with the normalized Betti number in the hard and sufficiently sharp spectral limit.
The numerical graph experiments focused on \(H_1\).
Across multiple initializations and multiple target regimes, optimization of the soft spectral objective produced systematic changes in the sampled hard normalized first Betti number.
The joint experiments further showed that the topological objective can be combined with an ordinary graph-statistical objective, such as degree variance, allowing topological and conventional structural constraints to be controlled simultaneously.

For Vietoris--Rips complexes, we considered heat- and resolvent-based Hodge spectral objectives and compared their optimization signals with persistence-based objectives.
The purpose of this comparison is not to replace persistent homology.
Persistent homology remains a natural representation when birth--death structure and multiscale persistence summaries are the primary objects of interest.
The present framework instead focuses on settings in which topology itself is used as an optimization signal.
In the controlled experiments considered here, Hodge spectral objectives produced more spatially distributed gradients and, in some cases, update directions that were more strongly aligned with prescribed geometric deformations.
These observations demonstrate that Hodge spectral losses and persistence-based losses can provide qualitatively different optimization information, but they should not be interpreted as establishing a general performance advantage of one approach over the other.

The theoretical construction is not restricted to first homology.
The soft Hodge operator, the \(W_q\)-weighted normalized spectral moment, and the corresponding backward computation are defined for arbitrary homological degree \(q\).
Thus, extending the framework from \(H_1\) to higher-dimensional homology does not require a fundamentally new theory.
However, practical optimization at higher degrees may require additional numerical and algorithmic care.
The dimensions of the chain spaces and the numbers of candidate simplices grow combinatorially with \(q\) and the problem size, increasing the cost of constructing and manipulating higher-order Hodge operators.

A second practical issue arises from the soft clique-simplex parameterization
\[
w_\sigma
=
\prod_{e\subset\sigma}p_e.
\]
As the simplex dimension increases, the number of probability factors in this product also increases.
When the individual edge probabilities are not close to one, the resulting simplex weight can become small, and the associated derivatives
\[
\frac{\partial w_\sigma}{\partial p_e}
=
\prod_{\substack{e'\subset\sigma\\e'\neq e}}p_{e'}
\]
may also become weak.
Thus, while the higher-dimensional extension is theoretically straightforward, its practical optimization may suffer from poor conditioning or attenuated gradients.
Possible remedies include alternative simplex-weight parameterizations, continuation or temperature schedules, gradient rescaling, adaptive optimization, and other conditioning techniques.
These issues concern the practical optimization of the framework rather than its underlying theoretical validity.

Computational scalability is another important limitation.
The current proof-of-concept implementation relies largely on dense ambient matrix operations, while the numbers of candidate simplices grow rapidly with the problem size.
Future implementations should exploit sparse linear algebra, polynomial or Chebyshev approximations, iterative spectral methods, and stochastic trace estimation in order to avoid full dense eigendecomposition or matrix-function evaluation.
The polynomial-trace formulation is particularly compatible with such approximations because it can be evaluated through repeated matrix--vector products rather than explicit spectral decomposition.

The same spectral formulation also suggests a possible longer-term connection to quantum computation.
Polynomial spectral filtering, trace estimation, and Betti-number estimation are related to several existing quantum linear-algebraic ideas.
This raises the possibility that parts of the present framework could eventually be combined with quantum spectral estimation or quantum trace-estimation procedures.
However, no quantum algorithm is implemented in this work, and no claim of quantum computational advantage is made.
The quantum connection should therefore be regarded only as a possible future computational direction.

From an application perspective, the framework is relevant to problems in which a structure must be \emph{designed} to satisfy a desired topology rather than merely analyzed after it has been constructed.
Potential examples include topology-constrained point-cloud and geometric design, graph and network generation with prescribed higher-order structure, and optimization of porous or cellular geometries with desired connectivity or cavity structure.
Related opportunities may arise in communication and sensor networks, material and structural design, and generative models whose outputs are required to satisfy global topological constraints.
In such settings, the Hodge spectral objective can act as a differentiable topological regularizer alongside conventional geometric, physical, or graph-statistical objectives.
The joint graph experiment provides a simple proof of concept for this type of multi-objective design.

In summary, the proposed framework establishes the chain
\[
\boxed{
\text{hard topology}
\;\longrightarrow\;
\text{low-frequency Hodge spectrum}
\;\longrightarrow\;
\text{differentiable optimization signal}.
}
\]
The theory provides a quantitative connection between hard Betti information and smooth spectral surrogates, while the numerical experiments demonstrate that these surrogates can serve as practical optimization objectives for graph clique complexes and Vietoris--Rips complexes.

\section*{Statements and Declarations}

\subsection*{Funding}
The authors declare that no funds, grants, or other support were received during the preparation of this manuscript.

\subsection*{Competing interests}
The authors have no relevant financial or non-financial interests to disclose.

\subsection*{Author contributions}
S.K. conceived the study, developed the theoretical framework, and carried out the numerical computations. Y.S. made significant contributions through extensive discussions with S.K. and provided important guidance for refining the direction of the study. All authors reviewed and approved the final manuscript.

\subsection*{Data availability}
The datasets generated and analyzed during the current study are available from the corresponding author upon reasonable request.

\subsection*{Code availability}
The code used for the numerical experiments is available from the corresponding author upon reasonable request.

\subsection*{Ethics approval}
Not applicable.

\subsection*{Consent to participate}
Not applicable.

\subsection*{Consent for publication}
Not applicable.

\begin{appendices}

\section{Backward Computation and Implementation Details}
\label{app:backward}

This appendix gives the backward computations used in the numerical experiments. We retain the full differentiation chain needed to reproduce the gradients, while omitting derivations that duplicate the main text and implementation variants not used in the principal experiments. Two closely related operators are used. The general ambient construction is
\[
\widehat L_q
=
\widetilde B_q^\top\widetilde B_q
+
\widetilde B_{q+1}\widetilde B_{q+1}^\top
+
\mu(I-W_q),
\]
whereas the graph polynomial-moment experiments use the penalty-free weighted-current operator
\[
L_q^{\mathrm{wc}}
=
\widetilde B_q^\top\widetilde B_q
+
\widetilde B_{q+1}\widetilde B_{q+1}^\top.
\]
For a scalar loss \(\mathcal J\) and a variable \(z\), we write
\(g_z=\partial\mathcal J/\partial z\).

\subsection{Soft simplex weights and weighted boundary operators}

For each candidate edge \(e\), let
\(p_e=\sigma(a_e)=(1+\exp(-a_e))^{-1}\).
For a simplex \(\sigma\) of dimension at least one,
\[
w_\sigma=\prod_{e\subset\sigma}p_e,
\]
while vertex weights are fixed to one. Define
\[
W_p
=
\operatorname{diag}(w_\sigma)_{\sigma\in K_{\max}^{(p)}},
\qquad
R_p=W_p^{1/2}
=
\operatorname{diag}(\rho_\sigma),
\qquad
\rho_\sigma=\sqrt{w_\sigma},
\]
and
\[
\widetilde B_p=R_{p-1}B_pR_p.
\]

For the weighted polynomial observable, an inactive-direction penalty is not required in the graph implementation. In the hard limit,
\[
L_q^{\mathrm{wc,hard}}=L_q(K)\oplus 0,
\qquad
W_q^{\mathrm{hard}}=I_{|K^{(q)}|}\oplus 0,
\]
and hence
\[
\operatorname{Tr}\!\left[
W_q^{\mathrm{hard}}
f(L_q^{\mathrm{wc,hard}})
\right]
=
\operatorname{Tr}f(L_q(K)).
\]
Thus inactive ambient simplex directions make no contribution to the weighted observable.

For polynomial filtering, let
\[
L_q^{\mathrm{ref}}
=
B_q^\top B_q+B_{q+1}B_{q+1}^\top,
\qquad
\Lambda_q
=
\lambda_{\max}(L_q^{\mathrm{ref}}),
\qquad
\eta_q=\Lambda_q^{-1}.
\]
Because the ambient candidate complex is fixed, \(d\Lambda_q=0\) during optimization. Moreover,
\[
L_q^{\mathrm{wc}}
=
R_q
\left[
B_q^\top W_{q-1}B_q
+
B_{q+1}W_{q+1}B_{q+1}^\top
\right]
R_q.
\]
Since \(0\preceq W_p\preceq I\), the bracketed matrix is bounded above by
\(L_q^{\mathrm{ref}}\), and therefore
\[
0\preceq L_q^{\mathrm{wc}}\preceq \Lambda_q I.
\]
Consequently,
\[
M=I-\eta_qL_q^{\mathrm{wc}}
\]
has spectrum in \([0,1]\), and no differentiation through a parameter-dependent spectral radius is required.

\subsection{Gradients of general low-pass spectral objectives}

Let \(L=L^\top\) and \(S_f(L)=\operatorname{Tr}f(L)\). For a differentiable scalar spectral function \(f\),
\[
dS_f
=
\operatorname{Tr}[f'(L)dL],
\qquad
\frac{\partial S_f}{\partial L}
=
f'(L).
\]
For the trace-matching loss
\[
\mathcal J_{\mathrm{trace}}
=
\frac12\left(S_f(L)-s^\star\right)^2,
\]
the matrix gradient is
\[
G_L
=
\frac{\partial\mathcal J_{\mathrm{trace}}}{\partial L}
=
\left(S_f(L)-s^\star\right)f'(L).
\]

For the heat filter
\(f_{\mathrm{heat}}(\lambda)=e^{-\lambda/\tau}\),
\[
f_{\mathrm{heat}}'(L)
=
-\frac1\tau e^{-L/\tau},
\qquad
G_L^{\mathrm{heat}}
=
-\frac{S_{\mathrm{heat}}(L)-s^\star}{\tau}
e^{-L/\tau}.
\]
For the resolvent filter
\(f_{\mathrm{res}}(\lambda)
=
\alpha_{\mathrm{res}}(\lambda+\alpha_{\mathrm{res}})^{-1}\),
\[
f_{\mathrm{res}}'(L)
=
-\alpha_{\mathrm{res}}
(L+\alpha_{\mathrm{res}}I)^{-2},
\]
and hence
\[
G_L^{\mathrm{res}}
=
-\left(S_{\mathrm{res}}(L)-s^\star\right)
\alpha_{\mathrm{res}}
(L+\alpha_{\mathrm{res}}I)^{-2}.
\]
Promotion or suppression objectives use the same inner derivatives multiplied by the derivative of the corresponding outer scalar objective.

\subsection{Matrix-valued spectral-filter matching}

For a known target operator \(L_{\mathrm{tar}}\), consider
\[
\mathcal J_{\mathrm{sf}}
=
\frac12
\left\|
f(L)-f(L_{\mathrm{tar}})
\right\|_F^2,
\qquad
D=f(L)-f(L_{\mathrm{tar}}).
\]
Let
\[
L=U\Lambda U^\top,
\qquad
\Lambda=\operatorname{diag}(\lambda_1,\ldots,\lambda_N).
\]
The Fr\'echet derivative of \(f\) is represented by the Loewner matrix
\[
(K_f)_{ij}
=
\begin{cases}
\dfrac{f(\lambda_i)-f(\lambda_j)}
{\lambda_i-\lambda_j},
& i\neq j,\\[2mm]
f'(\lambda_i),
& i=j.
\end{cases}
\]
Then
\[
U^\top(df(L))U
=
K_f\odot
\left(U^\top(dL)U\right).
\]
Using
\(d\mathcal J_{\mathrm{sf}}
=
\operatorname{Tr}[D^\top df(L)]\),
we obtain
\[
G_L
=
U
\left[
K_f\odot
\left(U^\top D U\right)
\right]
U^\top.
\]
For the resolvent filter, this can be written equivalently as
\[
G_L
=
-\alpha_{\mathrm{res}}
(L+\alpha_{\mathrm{res}}I)^{-1}
D
(L+\alpha_{\mathrm{res}}I)^{-1}.
\]
These gradients are propagated through the soft Hodge operator in the Vietoris--Rips experiments.

\subsection{\texorpdfstring{\(W_q\)-weighted polynomial spectral moment}{Wq-weighted polynomial spectral moment}}

Set
\[
L=L_q^{\mathrm{wc}},
\qquad
M=I-\eta_qL,
\qquad
S_{q,d}^{W}
=
\operatorname{Tr}(W_qM^d).
\]
The effective number of active \(q\)-simplices and the normalized moment are
\[
N_q^{\mathrm{eff}}
=
\operatorname{Tr}W_q,
\qquad
Z_q
=
N_q^{\mathrm{eff}}+\delta,
\qquad
\overline S_{q,d}^{W}
=
\frac{S_{q,d}^{W}}{Z_q}.
\]
For a prescribed soft target \(\bar s_q^\star\),
\[
\mathcal J_{\mathrm{top}}
=
\frac12
\left(
\overline S_{q,d}^{W}
-
\bar s_q^\star
\right)^2,
\qquad
r_q
=
\overline S_{q,d}^{W}
-
\bar s_q^\star.
\]
The target \(\bar s_q^\star\) is the target of the differentiable soft objective and is distinct from the sampled hard normalized Betti target used for evaluation.

\subsubsection{Differential of the weighted moment}

Starting from
\[
dS_{q,d}^{W}
=
\operatorname{Tr}[(dW_q)M^d]
+
\operatorname{Tr}[W_q\,d(M^d)],
\]
use
\[
d(M^d)
=
\sum_{k=0}^{d-1}
M^k(dM)M^{d-1-k},
\qquad
dM=-\eta_q\,dL.
\]
Cyclic invariance of the trace gives
\[
dS_{q,d}^{W}
=
\operatorname{Tr}[M^d dW_q]
-
\eta_q
\sum_{k=0}^{d-1}
\operatorname{Tr}
\left[
M^{d-1-k}
W_q
M^k
dL
\right].
\]
Define
\[
K_{q,d}^{W}
=
\frac1d
\sum_{k=0}^{d-1}
M^{d-1-k}
W_q
M^k.
\]
Because \(M\) and \(W_q\) are symmetric,
\((K_{q,d}^{W})^\top=K_{q,d}^{W}\), and therefore
\[
dS_{q,d}^{W}
=
\operatorname{Tr}[M^d dW_q]
-
d\eta_q
\operatorname{Tr}
[K_{q,d}^{W}dL].
\]
The first term is the direct weighting path \(W_q\to S_{q,d}^{W}\), whereas the second is the Hodge-operator path \(W_p\to L_q^{\mathrm{wc}}\to M\to S_{q,d}^{W}\). Both are included in the implementation.

\subsubsection{Gradient with respect to the soft Hodge operator}

When differentiating with respect to \(L\), \(W_q\) and \(Z_q\) are held fixed. Hence
\[
\frac{\partial S_{q,d}^{W}}{\partial L}
=
-d\eta_qK_{q,d}^{W},
\]
and
\[
G_L
=
\frac{\partial\mathcal J_{\mathrm{top}}}{\partial L}
=
-r_q
\frac{d\eta_q}{Z_q}
K_{q,d}^{W}.
\]
For \(W_q\neq I\), the averaged inserted kernel \(K_{q,d}^{W}\), rather than the unweighted factor \(M^{d-1}\), is the correct matrix derivative.

\subsubsection{Direct gradient with respect to the \(q\)-simplex weights}

Because
\(W_q=\operatorname{diag}(w_\sigma)_{\sigma\in K_{\max}^{(q)}}\),
\[
\operatorname{Tr}[M^d dW_q]
=
\sum_{\sigma\in K_{\max}^{(q)}}
(M^d)_{\sigma\sigma}\,dw_\sigma,
\]
while
\[
dZ_q
=
d\operatorname{Tr}W_q
=
\sum_{\sigma\in K_{\max}^{(q)}}dw_\sigma.
\]
Using
\(\overline S_{q,d}^{W}=S_{q,d}^{W}/Z_q\),
the direct contribution is
\[
\left.
\frac{\partial\overline S_{q,d}^{W}}
{\partial w_\sigma}
\right|_{\mathrm{direct}}
=
\frac{(M^d)_{\sigma\sigma}}{Z_q}
-
\frac{S_{q,d}^{W}}{Z_q^2}
=
\frac{
(M^d)_{\sigma\sigma}
-
\overline S_{q,d}^{W}
}{Z_q}.
\]
Therefore
\[
g_{w_\sigma}^{\mathrm{direct}}
=
\frac{r_q}{Z_q}
\left[
(M^d)_{\sigma\sigma}
-
\overline S_{q,d}^{W}
\right],
\qquad
\sigma\in K_{\max}^{(q)}.
\]
This term is absent from an unweighted trace objective and is essential for consistency with the \(W_q\)-weighted implementation.

\subsection{Backward propagation through the weighted Hodge operator}

Let
\[
A=\widetilde B_q,
\qquad
C=\widetilde B_{q+1},
\qquad
L=A^\top A+CC^\top.
\]
Then
\[
dL
=
(dA)^\top A
+
A^\top(dA)
+
(dC)C^\top
+
C(dC)^\top.
\]
Let \(G_L=\partial\mathcal J/\partial L\), symmetrized so that \(G_L^\top=G_L\). By matching
\(d\mathcal J=\operatorname{Tr}(G_L^\top dL)\),
\[
G_A=2AG_L,
\qquad
G_C=2G_LC.
\]
Thus
\[
G_{\widetilde B_q}
=
2\widetilde B_qG_L,
\qquad
G_{\widetilde B_{q+1}}
=
2G_L\widetilde B_{q+1}.
\]

If the full ambient operator
\[
\widehat L_q
=
L_q^{\mathrm{wc}}
+
\mu(I-W_q)
\]
is differentiated, the penalty contributes directly to the \(q\)-simplex weights:
\[
g_{w_\sigma}^{\mathrm{pen}}
=
-\mu(G_L)_{\sigma\sigma},
\qquad
\sigma\in K_{\max}^{(q)}.
\]
This contribution is absent from the penalty-free graph polynomial-moment implementation.

\subsection{From weighted boundaries to simplex weights}

Recall
\[
\widetilde B_p
=
R_{p-1}B_pR_p.
\]
In components,
\[
(\widetilde B_p)_{\tau\sigma}
=
\rho_\tau(B_p)_{\tau\sigma}\rho_\sigma.
\]
Let
\(G_{\widetilde B_p}
=
\partial\mathcal J/\partial\widetilde B_p\).
The contributions to the lower- and upper-dimensional square-root weights are
\[
g_{\rho_{p-1}}^{(p)}
=
(G_{\widetilde B_p}\odot B_p)\rho_p,
\qquad
g_{\rho_p}^{(p)}
=
(G_{\widetilde B_p}\odot B_p)^\top\rho_{p-1}.
\]
For the degree-\(q\) Hodge operator, the contributions from
\(\widetilde B_q\) and \(\widetilde B_{q+1}\) combine as
\[
g_{\rho_{q-1}}
=
(G_{\widetilde B_q}\odot B_q)\rho_q,
\]
\[
g_{\rho_q}
=
(G_{\widetilde B_q}\odot B_q)^\top\rho_{q-1}
+
(G_{\widetilde B_{q+1}}\odot B_{q+1})\rho_{q+1},
\]
and
\[
g_{\rho_{q+1}}
=
(G_{\widetilde B_{q+1}}\odot B_{q+1})^\top\rho_q.
\]
Since
\[
\rho_\sigma=\sqrt{w_\sigma},
\qquad
\frac{\partial\rho_\sigma}{\partial w_\sigma}
=
\frac1{2\rho_\sigma},
\]
the contribution through the weighted boundaries is
\[
g_{w_\sigma}^{\mathrm{bdry}}
=
\frac{g_{\rho_\sigma}}{2\rho_\sigma}.
\]

For the weighted polynomial moment, the total \(q\)-simplex gradient is
\[
g_{w_\sigma}
=
g_{w_\sigma}^{\mathrm{bdry}}
+
g_{w_\sigma}^{\mathrm{direct}},
\qquad
\sigma\in K_{\max}^{(q)}.
\]
For simplex dimensions \(p\neq q\),
\(g_{w_\sigma}=g_{w_\sigma}^{\mathrm{bdry}}\).
For a general ambient spectral loss, the penalty contribution is added whenever the term \(\mu(I-W_q)\) is present; if several mechanisms are combined, all applicable direct terms are included.

\subsection{From simplex weights to edge probabilities and logits}

For every simplex,
\[
w_\sigma
=
\prod_{e\subset\sigma}p_e.
\]
If \(e\subset\sigma\),
\[
\frac{\partial w_\sigma}{\partial p_e}
=
\prod_{\substack{e'\subset\sigma\\e'\neq e}}p_{e'},
\]
or, when \(p_e>0\),
\[
\frac{\partial w_\sigma}{\partial p_e}
=
\frac{w_\sigma}{p_e}.
\]
Hence
\[
g_{p_e}
=
\sum_{\sigma\supseteq e}
g_{w_\sigma}
\frac{\partial w_\sigma}{\partial p_e},
\]
where the sum includes all simplex dimensions entering the Hodge operator and the objective. Since
\[
p_e=\sigma(a_e),
\qquad
\frac{\partial p_e}{\partial a_e}
=
p_e(1-p_e),
\]
the edge-logit gradient is
\[
g_{a_e}
=
g_{p_e}p_e(1-p_e).
\]
For numerical stability, the implementation uses clipped probabilities or the product-of-other-edges form rather than division by a very small \(p_e\).

\subsection{Directional-derivative form used in the graph implementation}

The same derivatives are evaluated conveniently in directional form. Let
\(u=\{u_e\}\) be a direction in edge-logit space. Then
\[
\partial_up_e
=
p_e(1-p_e)u_e,
\]
and
\[
\partial_uw_\sigma
=
\sum_{e\subset\sigma}
\frac{\partial w_\sigma}{\partial p_e}
\partial_up_e.
\]
When all probabilities are nonzero, this is equivalently
\[
\partial_uw_\sigma
=
w_\sigma
\sum_{e\subset\sigma}
\frac{\partial_up_e}{p_e}.
\]
Furthermore,
\[
\partial_u\rho_\sigma
=
\frac{\partial_uw_\sigma}{2\rho_\sigma},
\]
and
\[
\partial_u\widetilde B_p
=
(\partial_uR_{p-1})B_pR_p
+
R_{p-1}B_p(\partial_uR_p).
\]

Writing
\(C_u=\partial_uL_q^{\mathrm{wc}}\),
\[
\begin{aligned}
C_u
={}&
(\partial_u\widetilde B_q)^\top\widetilde B_q
+
\widetilde B_q^\top(\partial_u\widetilde B_q)\\
&+
(\partial_u\widetilde B_{q+1})
\widetilde B_{q+1}^\top
+
\widetilde B_{q+1}
(\partial_u\widetilde B_{q+1})^\top.
\end{aligned}
\]
The weighted numerator derivative is
\[
\partial_uS_{q,d}^{W}
=
\operatorname{Tr}
[(\partial_uW_q)M^d]
-
d\eta_q
\operatorname{Tr}
[K_{q,d}^{W}C_u].
\]
The denominator derivative is
\[
\partial_uN_q^{\mathrm{eff}}
=
\operatorname{Tr}(\partial_uW_q).
\]
Therefore
\[
\partial_u\overline S_{q,d}^{W}
=
\frac{
\partial_uS_{q,d}^{W}
}{Z_q}
-
\frac{
S_{q,d}^{W}
}{Z_q^2}
\partial_uN_q^{\mathrm{eff}},
\]
and finally
\[
\partial_u\mathcal J_{\mathrm{top}}
=
r_q\,
\partial_u\overline S_{q,d}^{W}.
\]
The explicit term
\(\operatorname{Tr}[(\partial_uW_q)M^d]\)
is retained separately from the derivative through \(L_q^{\mathrm{wc}}\).

For the principal \(q=1\) graph experiments,
\[
W_1=\operatorname{diag}(p_e),
\qquad
R_1=\operatorname{diag}(\sqrt{p_e}),
\qquad
R_2=\operatorname{diag}(\sqrt{w_\triangle}),
\]
so
\[
\widetilde B_1=B_1R_1,
\qquad
\widetilde B_2=R_1B_2R_2,
\]
and
\[
L_1^{\mathrm{wc}}
=
\widetilde B_1^\top\widetilde B_1
+
\widetilde B_2\widetilde B_2^\top.
\]
With \(M=I-\eta_1L_1^{\mathrm{wc}}\),
\[
S_{1,d}^{W}
=
\operatorname{Tr}(W_1M^d)
=
\sum_ep_e(M^d)_{ee},
\qquad
N_1^{\mathrm{eff}}
=
\sum_ep_e.
\]
Thus
\[
\partial_uS_{1,d}^{W}
=
\sum_e
(\partial_up_e)(M^d)_{ee}
-
d\eta_1
\operatorname{Tr}
[K_{1,d}^{W}C_u],
\]
and
\[
\partial_uN_1^{\mathrm{eff}}
=
\sum_e\partial_up_e.
\]
These are the exact directional derivatives used in the main \(q=1\), \(n=15\) experiments.

\subsection{Joint optimization with soft degree variance}

For the auxiliary graph statistic, define
\[
d_i=\sum_{j\neq i}p_{ij},
\qquad
\bar d=\frac1n\sum_{i=1}^{n}d_i,
\qquad
V(p)
=
\frac1n
\sum_{i=1}^{n}
(d_i-\bar d)^2.
\]
For an edge \(e=(u,v)\),
\[
\frac{\partial V}{\partial p_e}
=
\frac2n
\left[
(d_u-\bar d)
+
(d_v-\bar d)
\right].
\]
Therefore
\[
\frac{\partial V}{\partial a_e}
=
\frac2n
\left[
(d_u-\bar d)
+
(d_v-\bar d)
\right]
p_e(1-p_e).
\]
For
\[
\mathcal J_{\mathrm{var}}
=
\frac12(V-V^\star)^2
\]
and
\[
\mathcal J_{\mathrm{joint}}
=
\mathcal J_{\mathrm{top}}
+
\lambda_{\mathrm{feat}}
\mathcal J_{\mathrm{var}},
\]
the edge-logit gradient is
\[
\nabla_a\mathcal J_{\mathrm{joint}}
=
\nabla_a\mathcal J_{\mathrm{top}}
+
\lambda_{\mathrm{feat}}
(V-V^\star)
\nabla_aV.
\]

\subsection{Backward computation for Vietoris--Rips point clouds}

For the Vietoris--Rips experiments, the variables are the point coordinates. At scale \(r\), define
\[
a_{ij}^{(r)}(X)
=
\frac{r-d_{ij}(X)}{\varepsilon},
\qquad
d_{ij}(X)
=
\sqrt{
\|x_i-x_j\|^2
+
\delta_{\mathrm{dist}}^2
},
\]
and
\[
p_{ij}^{(r)}
=
\sigma(a_{ij}^{(r)}).
\]
Suppose the spectral backward computation has produced
\[
g_{a_{ij}^{(r)}}
=
\frac{\partial\mathcal J}
{\partial a_{ij}^{(r)}}.
\]
Since
\[
\frac{\partial a_{ij}^{(r)}}{\partial d_{ij}}
=
-\frac1{\varepsilon},
\]
we obtain
\[
g_{d_{ij}}^{(r)}
=
-\frac1{\varepsilon}
g_{a_{ij}^{(r)}}.
\]
Moreover,
\[
\nabla_{x_i}d_{ij}
=
\frac{x_i-x_j}{d_{ij}}.
\]
Hence the point-coordinate gradient at scale \(r\) is
\[
\nabla_{x_i}\mathcal J^{(r)}
=
-\frac1{\varepsilon}
\sum_{j:(i,j)\in E_{\max}}
g_{a_{ij}^{(r)}}
\frac{x_i-x_j}{d_{ij}}.
\]
For the multi-scale objective
\[
\mathcal J(X)
=
\sum_{m=1}^{M}
\omega_m
\mathcal J_m(X),
\]
the gradients add:
\[
\nabla_{x_i}\mathcal J
=
-\frac1{\varepsilon}
\sum_{m=1}^{M}
\omega_m
\sum_{j:(i,j)\in E_{\max}}
g_{a_{ij}^{(r_m)}}
\frac{x_i-x_j}{d_{ij}}.
\]
Thus the heat, resolvent, and matrix-valued spectral objectives can all be propagated analytically to the point coordinates.

\subsection{Numerical stabilization}

The analytical derivations use
\(\rho_\sigma=\sqrt{w_\sigma}\).
Numerically, a small constant \(\epsilon_{\mathrm w}>0\) is added:
\[
\rho_\sigma
=
\sqrt{w_\sigma+\epsilon_{\mathrm w}},
\qquad
\frac{\partial\rho_\sigma}{\partial w_\sigma}
=
\frac1{2\sqrt{w_\sigma+\epsilon_{\mathrm w}}}.
\]
Expressions involving \(1/p_e\) are evaluated using clipped probabilities or the equivalent product-of-other-edges formula. Edge logits are also clipped to a finite interval to prevent overflow in the logistic function. These modifications affect only numerical stability and not the structure of the backward computation.

\end{appendices}


\bibliography{ref}

\end{document}